\documentclass[11pt]{amsart}

\usepackage{a4wide}
\usepackage[hyperfootnotes=false]{hyperref}
\usepackage{amssymb}
\usepackage{amsthm}
\usepackage{amsmath,bbm,bm}
\usepackage{mathrsfs}

\allowdisplaybreaks[2]
\numberwithin{equation}{section}

\newtheorem{theorem}{Theorem}[section]  
\newtheorem{lemma}[theorem]{Lemma}
\newtheorem{corollary}[theorem]{Corollary}
\newtheorem{proposition}[theorem]{Proposition}
\newtheorem{remark}[theorem]{Remark}
\newtheorem{definition}[theorem]{Definition}
\newtheorem{assumption}[theorem]{Assumption}

\newtheorem*{gammaRIE}{Property~$\boldsymbol{\gamma}$-\textup{(RIE)}}
\newtheorem*{RIE}{Property~\textup{(RIE)}}

\newcommand{\R}{\mathbb{R}}
\newcommand{\T}{\mathbb{T}}
\newcommand{\N}{{\mathbb{N}}}
\newcommand{\bX}{{\mathbf{X}}}
\newcommand{\X}{{\mathbb{X}}}
\newcommand{\1}{{\mathbf{1}}}
\newcommand{\dd}{{\,\mathrm{d}}}
\newcommand{\cL}{\mathcal{L}}
\renewcommand{\P}{\mathbb{P}}
\renewcommand{\d}{\mathrm{d}}

\title[A rough path approach to pathwise stochastic integration {\`a} la F{\"o}llmer]{A rough path approach to pathwise stochastic integration {\`a} la F{\"o}llmer}

\author[Das]{Purba Das}
\address{Purba Das, King's College London, United Kingdom}
\email{purba.das@kcl.ac.uk}

\author[Kwossek]{Anna P. Kwossek}
\address{Anna P. Kwossek, University of Vienna, Austria}
\email{anna.paula.kwossek@univie.ac.at}

\author[Pr{\"o}mel]{David J. Pr{\"o}mel}
\address{David J. Pr{\"o}mel, University of Mannheim, Germany}
\email{proemel@uni-mannheim.de}

\date{\today}

\begin{document}

\begin{abstract}
  We develop a general framework for pathwise stochastic integration that extends F{\"o}llmer's classical approach beyond gradient-type integrands and standard left-point Riemann sums and provides pathwise counterparts of It{\^o}, Stratonovich, and backward It{\^o} integration. More precisely, for a continuous path admitting both quadratic variation and L{\'e}vy area along a fixed sequence of partitions, we define pathwise stochastic integrals as limits of general Riemann sums and prove that they coincide with integrals defined with respect to suitable rough paths. Furthermore, we identify necessary and sufficient conditions under which the quadratic variation and the L{\'e}vy area of a continuous path are invariant with respect to the choice of partition sequences.
\end{abstract}

\maketitle 




\section{Introduction}

The classical theory of stochastic integration, built upon probabilistic concepts such as martingales and semimartingales, has proven to be highly successful in probability theory as well as its applications, such as the modeling of random phenomena across various disciplines. Nevertheless, over the past several decades, there has been an ongoing interest in developing a pathwise approach to stochastic integration, where integration is defined directly along individual trajectories, ideally independent of any underlying probability space, see e.g. \cite{Bichteler1981, Karandikar1995, Nutz2012, Davis2018}. This perspective is motivated by both theoretical and practical considerations: it allows for a finer analysis of irregular paths, facilitates comparisons between different stochastic models, and supports applications in settings where probabilistic assumptions may be too rigid or unavailable. For instance, recent developments in pathwise stochastic calculus have been stimulated, in particular, by rough path theory and by applications in mathematical finance under model uncertainty (often also referred to as Knightian uncertainty), see, e.g., \cite{Lyons1995,Follmer2013,Davis2014,Schied2018}.

As the first deterministic analogue to It{\^o}'s stochastic integration theory, in the seminal paper \cite{Follmer1981}, H. F{\"o}llmer initiated a pathwise integration theory based on a suitable notion of quadratic variation. More precisely, assuming that a continuous path $X \colon [0,T] \to \R^d$ possesses quadratic variation along a given sequence of partitions $\pi = (\pi^n)_{n \in \N}$ of $[0,T]$ with vanishing mesh size, he showed that for any twice continuously differentiable function $f$ the limit of left-point Riemann sums
\begin{equation}\label{eq: Follmer integral}
  \int_0^t \nabla f(X_s) \dd^\pi X_s := \lim_{n \to \infty} \sum_{[u,v] \in \pi^n } (\nabla f(X_u))^\top (X_{v \wedge t} - X_{u \wedge t}), \qquad t \in [0,T],
\end{equation}
exists, where $\nabla f$ denotes the gradient of $f$, and the integral satisfies a deterministic version of It{\^o}'s formula. The resulting integral $\int_0^t \nabla f(X_s) \dd^\pi X_s $ is often called \emph{F{\"o}llmer integral}. We recall that the sample paths of semimartingales, like Brownian motion and more general It{\^o} processes, do possess such quadratic variation, making the Föllmer integral naturally applicable to these stochastic processes. By now, the F{\"o}llmer integral has been generalized in various directions leading to new pathwise approaches to stochastic analysis, see e.g. \cite{Cont2010,Ananova2017,Cont2019,Chiu2022,Cont2024}, and has found numerous applications in mathematical finance under model uncertainty, where one seeks to avoid relying on a fixed underlying probabilistic model, see e.g.~\cite{Lyons1995,Davis2014,Schied2018,Cuchiero2019}.

In the present paper, we aim at generalizing the notion of the F{\"o}llmer integral in the sense that, for a given $\gamma \in [0,1]$ and a suitable class of not necessarily gradient-type integrands $Y$, the general pathwise integral
\begin{equation}\label{eq:intro}
  \int_0^t Y_s \dd^{\gamma,\pi} X_s := \lim_{n \to \infty} \sum_{[u,v] \in \pi^n} (Y_u + \gamma (Y_v - Y_u)) (X_{v \wedge t} - X_{u \wedge t}), \qquad t \in [0,T],
\end{equation}
exists as a uniform limit along the sequence $\pi = (\pi^n)_{n \in \N}$ of partitions. Notably, we shall allow, in particular, for integrands of the form $Y=f(X)$, with $f$ being a general twice continuously differentiable functions, and consider limits of approximating Riemann sums, where the integrand is given as a convex combination $\gamma Y_u + (1-\gamma) Y_v$ of the values of $Y_u$ and $Y_v$, with $[u,v]$ being a partition interval. We remark that $\gamma = 0$ corresponds to (forward) It{\^o}-type integration, $\gamma = \frac{1}{2}$ to Stratonovich-type integration and $\gamma = 1$ to backward It{\^o}-type integration; these being the most popular choices in probability theory.

To ensure the existence of the pathwise integral~\eqref{eq:intro} for non-gradient type integrands, the assumption of quadratic variation is generally not sufficient. It is therefore necessary to identify additional path properties of $X$ that restore the existence of the limit in \eqref{eq:intro}. To this end, we draw on rough path theory, initiated by Lyons~\cite{Lyons1998}, which provides a general framework for pathwise integration capable of handling paths of low regularity, such as the sample paths of semimartingales. In order to define a meaningful pathwise integral, rough path theory is founded on the insight that when a path lacks sufficient regularity, it must be enhanced to a so-called rough path containing information mimicking the iterated integrals of the path against itself, see e.g. \cite{Lyons2007,Friz2010,Friz2020}. Rather than relying on classical Riemann sums~\eqref{eq:intro}, the rough integral is constructed as the limit of compensated Riemann sums, which incorporate higher-order terms involving the iterated integrals. The rough path extension of a path is precisely what ensures a well-posed integration map leading, e.g., to stability estimates and well-posedness of differential equations driven by irregular signals.

We show that a sufficient condition for the existence of the pathwise integral~\eqref{eq:intro} is that the continuous path~$X$ admits both quadratic variation and L{\'e}vy area along the sequence $\pi = (\pi^n)_{n \in \N}$ of partitions, together with suitable regularity conditions. It turns out that these assumptions can be equivalently formulated as a path property, called Property $\gamma$-\textup{(RIE)}, which ensures that the path $X$ can be canonically enhanced to a rough path as a limit of Riemann sums. Assuming Property $\gamma$-\textup{(RIE)}, we prove that the associated rough integral exists as a limit of general Riemann sums~\eqref{eq:intro}, thus establishing a link between rough and F{\"o}llmer integration. Property $\gamma$-\textup{(RIE)} generalizes the original Property \textup{(RIE)} (this coincides with $\gamma$-\textup{(RIE)} for $\gamma=0$), as first introduced in~\cite{Perkowski2016} for continuous paths and extended to c{\`a}dl{\`a}g paths in~\cite{Allan2024}. The original Property \textup{(RIE)} allows to recover the rough integral as a limit of left-point Riemann sums  and found applications in model-free finance \cite{Allan2023a} and numerical analysis \cite{Allan2023c}. Notably, for $\gamma =1/2$ -- corresponding to the Stratonovich-type setting -- Property $\gamma$-\textup{(RIE)} is equivalent to the existence of L{\'e}vy area along a sequence of partitions and suitable regularity conditions, which implies that the pathwise integral exists, as expected, even without assuming the existence of quadratic variation.

Unlike rough integrals or, more classically, Riemann--Stieltjes and Young integrals, the F{\"o}llmer integral~\eqref{eq: Follmer integral} depends on the choice of the sequence of partitions since the underlying notion of quadratic variation is dependent on the chosen sequence of partitions. Indeed, as shown by Freedman \cite[p.~47]{Freedman1983}, for any continuous function, one can construct a sequence of partitions such that the resulting quadratic variation is zero. This result was extended by \cite{Davis2018}, where it is established that, given any sufficiently oscillating continuous path and any increasing function, one can construct a Lebesgue-type sequence of partitions such that the resulting quadratic variation coincides with the given increasing function. These observations motivated to study the invariance of quadratic variation with respect to the choice of partition sequences, as initiated by \cite{cont-das2022,contdas2023}. In particular, \cite{contdas2023} introduces the concept of quadratic roughness for H{\"o}lder continuous paths, which provides a necessary and sufficient condition under which the quadratic variation is invariant with respect to the choice of partition sequences.

We first generalize the concept of quadratic roughness, allowing for continuous paths of finite $p$-variation for $p\in [2,3)$. This broader class of paths has the advantage of containing the sample paths of continuous semimartingales, in contrast to the space of H{\"o}lder continuous functions. Secondly, we identify a necessary and sufficient condition -- referred to as L{\'e}vy roughness -- under which the L{\'e}vy area of a continuous path is invariant with respect to the choice of partition sequences.

Finally, we briefly demonstrate that the deterministic framework we developed for continuous paths can be applied to the sample paths of various stochastic processes, which in particular allows to recover the classical It{\^o} and Stratonovich integrals. This establishes a direct link between the pathwise integral~\eqref{eq:intro} {\`a} la F{\"o}llmer and classical stochastic integral.

\smallskip
\noindent \textbf{Organization of the paper:} In Section~\ref{sec: the rough integral as a limit of general Riemann sums} we develop the pathwise integration theory {\`a} la F{\"o}llmer, based on limits of general Riemann sums. In Section~\ref{sec: invariance} we establish necessary and sufficient conditions under which the quadratic variation and the L{\'e}vy area of a continuous path are invariant with respect to the choice of partition sequences. In Section~\ref{sec: application to stochastic integration} we explore the connections between the pathwise framework and stochastic processes as well as classical stochastic integration.

\smallskip
\noindent \textbf{Acknowledgement:} A. P. Kwossek was affiliated with the University of Mannheim for the majority of this project's duration.

\section{Pathwise `stochastic' integration}\label{sec: the rough integral as a limit of general Riemann sums}

In this section, we establish pathwise integration in the spirit of F{\"o}llmer, allowing to recover various types of classical stochastic integration, like It{\^o}, Stratonovich, and backward It{\^o} integration. To make the presentation more accessible, we begin by giving a brief exposition of the constructed pathwise integrals from the perspective of F{\"o}llmer integration. The more general construction of pathwise `stochastic' integration is then based on rough path theory.

\subsection{F{\"o}llmer-type integration: beyond gradients and left-point Riemann sums}

We aim at generalizing the notion of F{\"o}llmer integration in two directions, cf. \cite{Follmer1981}. First, we allow for integrands which are not necessarily the gradient of a function, as required in the classical construction of F\"ollmer. Second, we define the pathwise integral as a limit of general Riemann sums instead of only considering approximation by left-point Riemann sums.

To that end, we let $T > 0$ be a fixed finite time horizon. We let $C([0,T];\R^d)$ denote the space of continuous paths $X \colon [0,T] \to \R^d$, and write $C^k(\R^d;\R^{m \times d})$ for the space of $k$-times continuously differentiable functions $f\colon \R^d \to \R^{m \times d}$, for $n \in \N$. A \emph{partition} $\mathcal{P}$ of an interval $[s,t]$ is a finite set of points between and including the points $s$ and $t$, i.e., $\mathcal{P} = \{s = u_0 < u_1 < \cdots < u_N = t\}$ for some $N \in \N$, and its mesh size is denoted by $|\mathcal{P}|:= \max\{|u_{i+1} - u_i| \, : \, i = 0, \ldots, N-1\}$.

Moreover, for a path $X \in C([0,T];\R^d)$, we will often use the shorthand notation:
\begin{equation*}
	X_{s,t} := X_t - X_s, \qquad \text{for} \quad 0 \leq s \leq t \leq T.
\end{equation*}
We make the following assumption on the considered sequences of partitions.

\begin{assumption}\label{ass: partitions}
  Let $X \in C([0,T];\R^d)$ and $\pi = (\pi^n)_{n\in \N}$, with $\pi^n = \{0 = t^n_0 < t^n_1 < \dots < t^n_{N_n} = T\}$, $n \in \N$, be a sequence of partitions of the interval $[0,T]$ such that $\sup\{|X_{t^n_k,t^n_{k+1}}| \,: k=0, \dots, N_n-1\}$ converges to $0$ as $n \to \infty$.
\end{assumption}

Note that, for $X \in C([0,T];\R^d)$, any sequence $\pi = (\pi^n)_{n\in \N}$ of partitions with vanishing mesh size, i.e. $|\pi^n|\to 0$ as $n\to \infty$, satisfies $\sup\{|X_{t^n_k,t^n_{k+1}}| \, : k=0, \dots, N_n-1\}\to 0$ as $n \to \infty$.

For a given $\gamma \in [0,1]$ and $f \in C^3(\R^d;\R^{m \times d})$, we define the pathwise integral
\begin{equation}\label{eq: general Follmer integral}
  \int_0^t f(X_s) \dd^{\gamma,\pi} X_s := \lim_{n \to \infty}  \sum_{[u,v] \in \pi^n} (f(X_v) + \gamma (f(X_v) - f(X_u)) (X_{v \wedge t} - X_{u \wedge t}),
\end{equation}
for $t \in [0,T]$, where the limit is taken in the sense of uniform convergence. We note that $\gamma = 0$, $\gamma = \frac{1}{2}$ and $\gamma = 1$ correspond to (forward) It{\^o}-type integration, Stratonovich-type integration and backward It{\^o}-type integration respectively.

\medskip

To ensure the existence of the limit in \eqref{eq: general Follmer integral} for suitable functions~$f$, F{\"o}llmer introduced the concept of pathwise quadratic variation along a sequence of partitions, see \cite{Follmer1981}.

\begin{definition}[Quadratic variation of a path along a sequence of partitions]\label{def: quad variation}
  Let $X$ and $\pi = (\pi^n)_{n\in \N}$ be as in Assumption~\ref{ass: partitions}. We say that a path $X \in C([0,T];\R^d)$ possesses \emph{quadratic variation} along $\pi$ if
  \begin{equation*}
    [X]^\pi_t :=  \lim_{n \to \infty} [X]^{\pi^n}_t := \lim_{n \to \infty} \sum_{k=0}^{N_n-1} X_{t^n_k \wedge t, t^n_{k+1} \wedge t} \otimes X_{t^n_k \wedge t, t^n_{k+1} \wedge t}, \qquad t \in [0,T],
  \end{equation*}
  exists, where the convergence is uniform in $t \in [0,T]$. We denote by $Q^\pi([0,T];\R^d)$ the set of paths $X \in C([0,T];\R^d)$ possessing quadratic variation along $\pi$, where we implicitly assume that $X$ and $\pi$ satisfy Assumption~\ref{ass: partitions}.
\end{definition}

Let us remark that in the present setting, as shown in, e.g., \cite{Perkowski2016,Chiu2018}, the definition of quadratic variation provided in Definition~\ref{def: quad variation} is equivalent to the definition of quadratic variation introduced in the work of F\"ollmer~\cite{Follmer1981}. Assuming the existence of quadratic variation and that $f= \nabla g$ for some twice continuously differentiable function~$g$, the F{\"o}llmer integral~\eqref{eq: Follmer integral} exists and satisfies a pathwise It{\^o} formula, see~\cite[Th{\'e}or{\`e}me]{Follmer1981}.

\medskip

To generalize the pathwise integral~\eqref{eq: general Follmer integral} allowing for general functions~$f$, we introduce an additional path property. More precisely, we additionally impose the existence of the so-called L{\'e}vy area of a path along the given sequence of partitions and assume some regularity properties. To that end, we first recall the concept of control functions. Let $\Delta_T := \{(s,t) \in [0,T]^2 \, : \, s \leq t\}$ denote the standard $2$-simplex and let $C(\Delta_T;[0,\infty))$ be the space of all continuous functions $c \colon \Delta_T \to [0,\infty)$. A function $c\in C(\Delta_T;[0,\infty))$ is called a \emph{control function} if it is superadditive, in the sense that $c(s,u) + c(u,t) \leq c(s,t)$ for all $0 \leq s \leq u \leq t \leq T$. For instance, $c(s,t)=|t-s|$ is a typical example of a control function. For two vectors $x = (x^1, \ldots, x^d)$, $y = (y^1, \ldots, y^d) \in \R^d$ we use the standard tensor product
\begin{equation*}
  x \otimes y := (x^i y^j)_{i, j = 1, \ldots, d} \in \R^{d \times d}.
\end{equation*}

\begin{definition}[L{\'e}vy area of a path along a sequence of partitions]
  Let $X$ and $\pi = (\pi^n)_{n\in \N}$ be as in Assumption~\ref{ass: partitions}. We say that a path $X \in C([0,T];\R^d)$ possesses \emph{L{\'e}vy area} along $\pi$ if
  \begin{equation*}
    \cL^\pi_t(X) := \lim_{n \to \infty} \cL^{\pi^n}_t(X) := \lim_{n \to \infty} \sum_{k=0}^{N_n-1} (X_{t^n_k \wedge t} + X_{t^n_{k+1} \wedge t}) \otimes X_{t^n_k \wedge t, t^n_{k+1} \wedge t}, \qquad t \in [0,T],
  \end{equation*}
  exists, where the convergence is uniform in $t \in [0,T]$. We denote by $\cL^\pi([0,T];\R^d)$ the set of $X \in C([0,T];\R^d)$ possessing L{\'e}vy area along $\pi$, where we implicitly assume that $X$ and $\pi$ satisfy Assumption~\ref{ass: partitions}.
\end{definition}

Assuming the existence of quadratic variation and L{\'e}vy area along a given sequence of partitions, we establish the existence of the pathwise integral~\eqref{eq: general Follmer integral} for general functions in the next proposition under the following regularity conditions.

\begin{assumption}[Regularity conditions]\label{ass: regularity}
  Let $X$ and $\pi = (\pi^n)_{n\in \N}$ be as in Assumption~\ref{ass: partitions}. There exists a constant $p \in (2,3)$ and a control function $c$ such that
  \begin{equation*}
    \sup_{(s,t) \in \Delta_T} \frac{|X_{s,t}|^p}{c(s,t)} + \sup_{n \in \N} \sup_{0 \le k < \ell \le N_n} \frac{|\cL^{\pi^n}_{t^n_k,t^n_\ell}(X) - (X_{t^n_k} + X_{t^n_\ell}) \otimes X_{t^n_k,t^n_\ell}|^{\frac{p}{2}}}{c(t^n_k,t^n_\ell)} \lesssim 1,
  \end{equation*}
  where $\cL^{\pi^n}_{s,t}(X) := \cL^{\pi^n}_t(X) - \cL^{\pi^n}_s(X)$.
\end{assumption}

\begin{proposition}\label{prop:generalized Follner integration}
  Let $X$ and $\pi = (\pi^n)_{n\in \N}$ be as in Assumption~\ref{ass: partitions}. Suppose that $X$ possesses quadratic variation and L{\'e}vy area, both along $\pi$, and that Assumption~\ref{ass: regularity} is satisfied. Then, for every $\gamma \in [0,1]$ and for $f\in C^2(\R^d;\R^{m \times d})$, the pathwise integral
  \begin{equation*}
    \int_0^t f(X_s) \dd^{\gamma,\pi} X_s := \lim_{n \to \infty}  \sum_{k=0}^{N_n-1} (f(X_{t^n_k}) + \gamma (f(X_{t^n_{k+1}}) - f(X_{t^n_k})) (X_{t^n_{k+1} \wedge t} - X_{t^n_k \wedge t}),
  \end{equation*}
  for $t \in [0,T]$, exists, where the limit is uniform in $t \in [0,T]$. Moreover, for $F\in C^3(\R^d;\R^m)$, the pathwise F\"ollmer--It\^o formula
  \begin{equation*}
    F(X_t)-F(X_0)= \int_0^t \textup{D} F(X_s) \dd^{\gamma,\pi} X_s + \Big(\frac{1}{2}-\gamma\Big) \int_0^t \textup{D}^2 F(X_s) \dd [X]^\pi_s,\quad  t\in [0,T],
  \end{equation*}
  holds, where $\textup{D}$, $\textup{D}^2:$ $\R^{m} \mapsto \R^{m\times d}$ denote the differential operators of order $1$ and $2$, respectively.
\end{proposition}

\begin{proof}
  Since $X$ possesses quadratic variation and L{\'e}vy area, both along $\pi$, and $X$ satisfies Assumption~\ref{ass: regularity}, $X$ satisfies Property $\gamma$-\textup(RIE), by Lemma~\ref{lem: RIE eq to gamma RIE} and~\ref{lem: RIE eq to Levy and quad var}. We refer to Subsection~\ref{subsec:rough integration} for the definition of Property $\gamma$-\textup(RIE). Moreover, note that $f(X)$ is a controlled path w.r.t. $X$, see e.g. \cite[Example~4.7]{Perkowski2016}. Hence, the existence of the pathwise integral is an immediate consequence of Theorem~\ref{thm: general pathwise stochastic integral}, and the pathwise It\^o formula follows by \cite[Corollary~3.13]{Imkeller2015}.
\end{proof}

Let us conclude this subsection with some remarks.

\begin{remark}
  In Theorem~\ref{thm: general pathwise stochastic integral} below, we not only show that the integral $\int_0^t f(X_s) \dd^{\gamma,\pi} X_s$ exists but also that it coincides with the rough integral with respect to a suitably defined rough path. Consequently, the powerful continuity and stability properties of rough integration transfer to the integral $\int_0^t f(X_s) \dd^{\gamma,\pi} X_s$. Moreover, let us remark that the existence of quadratic variation is not necessary for the claim of Proposition~\ref{prop:generalized Follner integration} to hold true in case of $\gamma=\frac{1}{2}$, and that Proposition~\ref{prop:generalized Follner integration} can, for instance, be generalized to non-anticipative, path-dependent functionals $f$ with suitable differentiablity properties in the sense of Dupire~\cite{Dupire2019}, see also \cite{Cont2010}.
\end{remark}

\begin{remark}\label{rem: Levy area}
  In probability theory, the L{\'e}vy area was originally introduced as the area that is enclosed by any trajectory of the Brownian motion $(B^1, B^2)$ and its chord, see~\cite{Levy1940}. It is defined by $\frac{1}{2} (\int_0^T B^1_t \dd B^2_t - \int_0^T B^2_t \dd B^1_t)$, which makes sense
  in terms of stochastic integration.

  In present pathwise framework, assuming the respective limits exist, indeed, we obtain that
  \begin{align*}
    &\cL^\pi_t(X)^{ij} - (X_t^i X_t^j - X_0^i X_0^j) \\
    &\quad = \lim_{n \to \infty} \sum_{k=0}^{N_n-1} (X^i_{t^n_k \wedge t} + X^i_{t^n_{k+1} \wedge t}) X^j_{t^n_k \wedge t, t^n_{k+1} \wedge t} - \sum_{k=0}^{N_n-1} \bigg(X^i_{t^n_{k+1} \wedge t} X^j_{t^n_{k+1} \wedge t} - X^i_{t^n_k \wedge t} X^j_{t^n_k \wedge t} \bigg) \\
    &\quad = \lim_{n \to \infty} \sum_{k=0}^{N_n-1} (X^i_{t^n_k} + \gamma X^i_{t^n_k, t^n_{k+1}}) X^j_{t^n_k \wedge t, t^n_{k+1} \wedge t} - \sum_{k=0}^{N_n-1} (X^j_{t^n_k} + \gamma X^j_{t^n_k, t^n_{k+1}}) X^i_{t^n_k \wedge t, t^n_{k+1} \wedge t} \\
    &\quad = \lim_{n \to \infty} \int_0^t X^i_s \dd^{\gamma,\pi^n} X^j_s - \int_0^t X^j_s \dd^{\gamma,\pi^n} X^i_s \\
    &\quad = \int_0^t X^i_s \dd^{\gamma,\pi} X^j_s - \int_0^t X^j \dd^{\gamma,\pi} X^i_s,
  \end{align*}
  for every $i, j = 1, \ldots, d$, which thus corresponds to the classical notion of the L{\'e}vy area. We further note that
  \begin{equation*}
    \int_0^t X^i_s \dd^{\gamma,\pi^n} X^j_s - \int_0^t X^j_s \dd^{\gamma,\pi^n} X^i_s = \int_0^t X^i_s \dd^{0,\pi^n} X^j_s - \int_0^t X^j_s \dd^{0,\pi^n} X^i_s,
  \end{equation*}
  that is, in terms of general Riemann sums, the pathwise L{\'e}vy area is invariant with respect to the choice of $\gamma$ and coincides with the It{\^o}-type one.
\end{remark}

\subsection{Essentials on rough path theory}

In this subsection we introduce some standard notation and recall the essentials from the theory of rough paths. For a more detailed exposition of rough path theory, we refer, e.g., to \cite{Lyons2007,Friz2010,Friz2020}.

\medskip

We shall write $a \lesssim b$ to mean that there exists a constant $C > 0$ such that $a \leq C b$. The constant $C$ may depend on the normed space, e.g.~through its dimension or regularity parameters. For two vector spaces, the space of linear maps from $E_1 \to E_2$ is denoted by $\cL(E_1;E_2)$.

\medskip

For a normed space $(E,|\cdot|)$, we let $C([0,T];E)$ denote the space of continuous paths from $[0,T]$ to $E$. For $p \geq 1$, the \emph{$p$-variation of a path} $X \in C([0,T];E)$ is given by
\begin{equation*}
  \|X\|_p := \|X\|_{p,[0,T]} \qquad \text{with} \qquad \|X\|_{p,[s,t]} := \bigg(\sup_{\mathcal{P}\subset[s,t]} \sum_{[u,v]\in \mathcal{P}} |X_v - X_u|^p \bigg)^{\frac{1}{p}}, \quad (s,t) \in \Delta_T,
\end{equation*}
where the supremum is taken over all possible partitions $\mathcal{P}$ of the interval $[s,t]$. We recall that, given a path $X\in C([0,T];E)$, we have that $\|X\|_p < \infty$ if and only if there exists a control function $c$ such that\footnote{Here and throughout, we adopt the convention that $\frac{0}{0} := 0$.}
\begin{equation*}
  \sup_{(u,v) \in \Delta_T} \frac{|X_v - X_u|^p}{c(u,v)} < \infty.
\end{equation*}
We write $C^{p\text{-var}} = C^{p\text{-var}}([0,T];E)$ for the space of paths $X \in C([0,T];E)$ which satisfy $\|X\|_p < \infty$. Moreover, for a path $X \in C([0,T];\R^d)$, we will often use the shorthand notation:
\begin{equation*}
  X_{s,t} := X_t - X_s, \qquad \text{for} \quad (s,t) \in \Delta_T.
\end{equation*}
For $r \geq 1$ and a two-parameter function $\X \colon \Delta_T \to E$, we similarly define
\begin{equation*}
  \|\X\|_r := \|\X\|_{r,[0,T]} \qquad \text{with} \qquad \|\X\|_{r,[s,t]} := \bigg(\sup_{\mathcal{P} \subset [s,t]} \sum_{[u,v] \in \mathcal{P}} |\X_{u,v}|^r\bigg)^{\frac{1}{r}}, \quad (s,t) \in \Delta_T.
\end{equation*}
We write $C_2^{r\text{-var}} = C_2^{r\text{-var}}(\Delta_T;E)$ for the space of continuous functions $\X \colon \Delta_T \to E$ which satisfy $\|\X\|_r < \infty$.

\medskip 

For $p \in [2,3)$, a pair $\bX = (X,\X)$ is called a \emph{(continuous) rough path} over $\R^d$ if
\begin{enumerate}
  \item[(i)] $X \in C^{p\text{-var}}([0,T];\R^d)$ and $\X \in C_2^{\frac{p}{2}\text{-var}}(\Delta_T;\R^{d \times d})$, and
  \item[(ii)] Chen's relation: $\X_{s,t} = \X_{s,u} + \X_{u,t} + X_{s,u} \otimes X_{u,t}$ holds for all $0 \leq s \leq u \leq t \leq T$.
\end{enumerate}
In component form, condition (ii) states that $\X^{ij}_{s,t} = \X^{ij}_{s,u} + \X^{ij}_{u,t} + X^i_{s,u} X^j_{u,t}$ for every $i$ and $j$. We will denote the space of $p$-rough paths by $\mathcal{C}^p = \mathcal{C}^p([0,T];\R^d)$. On the space $\mathcal{C}^p([0,T];\R^d)$, we use the natural seminorm
\begin{equation*}
  \|\bX\|_{p} := \|\bX\|_{p,[0,T]} \qquad \text{with} \qquad \|\bX\|_{p,[s,t]} := \|X\|_{p,[s,t]} + \|\X\|_{\frac{p}{2},[s,t]}
\end{equation*}
for $(s,t) \in \Delta_T$.

\medskip

Let $p \in (2,3)$ and $q > 0$ such that $\frac{2}{p} + \frac{1}{q} > 1$, and $X \in C^{p\text{-var}}([0,T];\R^d)$. We say that a pair $(Y,Y')$ is a \emph{controlled path} (with respect to $X$), if
\begin{equation*}
  Y \in C^{p\text{-var}}([0,T];\R^{d\times n}), \quad Y' \in C^{q\text{-var}}([0,T];\cL(\R^d;\R^{d\times n})), \quad \text{and} \quad R^Y \in C^{r\text{-var}}_2(\Delta_T;\R^{d\times n}),
\end{equation*}
where $R^Y$ is defined by
\begin{equation*}
  Y_{s,t} = Y'_s X_{s,t} + R^Y_{s,t} \qquad \text{for all} \quad (s,t) \in \Delta_T,
\end{equation*}
and $\frac{1}{r} = \frac{1}{p} + \frac{1}{q}$. We write $\mathscr{C}^{p,q}_X = \mathscr{C}^{p,q}_X([0,T];\R^{d\times n})$ for the space of $\R^{d\times n}$-valued controlled paths, which becomes a Banach space when equipped with the norm
\begin{equation*}
  (Y,Y') \mapsto |Y_0| + |Y'_0| + \|Y'\|_{q,[0,T]} + \|R^Y\|_{r,[0,T]}.
\end{equation*}

\medskip

Given $p \in (2,3)$, $\bX = (X,\X) \in \mathcal{C}^p([0,T];\R^d)$ and $(Y,Y') \in \mathscr{C}_X^{p,q}([0,T];\cL(\R^d;\R^{d\times n}))$, the (forward) rough integral
\begin{equation}\label{eq: rough integration}
  \int_s^t Y_r \dd \bX_r :=  \lim_{|\mathcal{P}| \to 0} \sum_{[u,v] \in \mathcal{P}} (Y_u X_{u,v} + Y'_u \X_{u,v}), \qquad (s,t) \in \Delta_T,
\end{equation}
exists (in the classical mesh Riemann--Stieltjes sense), where the limit is taken along any sequence $(\mathcal{P}^n)_{n\in \N}$ of partitions of the interval $[s,t]$ such that $|\mathcal{P}^n| \to 0$ as $n\to \infty$. More precisely, in writing the product $Y_u X_{u,v}$, we apply the operator $Y_u \in \cL(\R^d;\R^m)$ onto $X_{u,v} \in \R^d$; and in writing the product $Y'_u \X_{u,v}$, we use the natural identification of $\cL(\R^d;\cL(\R^d;\R^m))$ with $\cL(\R^d \otimes \R^d;\R^m)$. Moreover, the rough integral comes with the estimate
\begin{equation*}
  \bigg|\int_s^t Y_u \dd \bX_u - Y_s X_{s,t} - Y'_s \X_{s,t}\bigg| \leq C \Big( \|R^Y\|_{r,[s,t]} \|X\|_{p,[s,t]} + \|Y'\|_{q,[s,t]} \|\X\|_{\frac{p}{2},[s,t]} \Big)
\end{equation*}
for some constant $C$ depending only on $p$, $q$ and $r$; see e.g.~\cite[Theorem~4.9]{Perkowski2016}.

\subsection{Rough integration as limit of general Riemann sums}\label{subsec:rough integration}

In this subsection we aim to recover the rough integral~\eqref{eq: rough integration} in the spirit of F{\"o}llmer integral, i.e. as a limit of (general) Riemann sums along a sequence of partitions. To that end, we recall that $X \in C([0,T];\R^d)$ and $\pi = (\pi^n)_{n\in \N}$, with $\pi^n = \{0 = t^n_0 < t^n_1 < \dots < t^n_{N_n} = T\}$, $n \in \N$, is a sequence of partitions of the interval $[0,T]$ such that Assumption~\ref{ass: partitions} is satisfied, that is, $\sup\{|X_{t^n_k,t^n_{k+1}}| \,: k=0, \dots, N_n-1\} \to 0$ as $n \to \infty$. For a given $\gamma \in [0,1]$ and a (suitable) integrand $Y\in C([0,T];\R^{d\times n})$, we define a pathwise integral by
\begin{equation}\label{eq: general pathwise integral}
  \int_0^t Y_s \dd^{\gamma,\pi} X_s := \lim_{n \to \infty} \int_0^t Y_s  \dd^{\gamma,\pi^n} X_s, \qquad t \in [0,T],
\end{equation}
where the limit is taken in uniform convergence and
\begin{equation*}
  \int_0^t Y_s \dd^{\gamma,\pi^n} X_s := \sum_{[u,v] \in \pi^n} (Y_u + \gamma (Y_v - Y_u))  (X_{v \wedge t} - X_{u \wedge t}).
\end{equation*}
To ensure the existence of the limit in \eqref{eq: general pathwise integral}, we postulate the following property.

\begin{gammaRIE}
  Let $X$ and $\pi = (\pi^n)_{n\in \N}$ be as in Assumption~\ref{ass: partitions}, and let $\gamma \in [0,1]$. We assume that
  \begin{enumerate}
      \item[(i)] the Riemann sums
  \begin{equation*}
    \int_0^t X_s \otimes \d^{\gamma, \pi^n} X_s := \sum_{k=0}^{N_n-1} (X_{t^n_k} + \gamma X_{t^n_k,t^n_{k+1}}) \otimes X_{t^n_k \wedge t, t^n_{k+1} \wedge t}, \quad t \in [0,T],
  \end{equation*}
  converge uniformly as $n \to \infty$ to a limit, which we denote by $\int_0^t X_s \otimes \d^{\gamma,\pi} X_s$,
    \item[(ii)] there exists a constant $p \in (2,3)$ and a control function $c$ such that
  \begin{equation}\label{eq: bounded p/2 variation assumption}
    \sup_{(s,t) \in \Delta_T} \frac{|X_{s,t}|^p}{c(s,t)} + \sup_{n \in \N} \sup_{0 \le k < \ell \le N_n} \frac{| (\int_0^{\cdot} X_s \otimes \d^{\gamma,\pi^n} X_s)_{t^n_k,t^n_\ell} - (X_{t^n_k} + \gamma X_{t^n_k,t^n_{\ell}}) \otimes X_{t^n_k,t^n_\ell}|^{\frac{p}{2}}}{c(t^n_k,t^n_\ell)} \lesssim 1.
  \end{equation}
  \end{enumerate}
\end{gammaRIE}
We say that a path $X \in C([0,T];\R^d)$ satisfies \emph{Property~$\gamma$-\textup{(RIE)}} relative to $\gamma$, $p$ and $\pi$ if $\gamma$, $p$, $\pi$ and $X$ together satisfy Property $\gamma$-\textup{(RIE)}.

\begin{remark}
  Property $\gamma$-\textup{(RIE)} is closely related to Property \textup{(RIE)}, as introduced in~\cite{Perkowski2016} and~\cite{Allan2024}, which recovers the rough integral as a limit of (non-compensated) left-point Riemann sums, see~\cite[Theorem~4.19]{Perkowski2016}. Moreover, Property $\gamma$-\textup{(RIE)} is, roughly speaking, equivalent to the existence of quadratic variation and L{\'e}vy area together with the regularity conditions in Assumption~\ref{ass: regularity}. A precise study on the relation of these various assumptions is provided in Subsection~\ref{sec: on the assumption of general Riemann integrals} below.
\end{remark}

Under Property $\gamma$-\textup(RIE), we shall derive the existence of the pathwise integral~\eqref{eq: general pathwise integral} and show that it coincides with the rough integral, i.e. the rough integral can be approximated by non-compensated, general Riemann sums along suitable sequences of partitions.

\medskip

To properly define the rough integral assuming Property $\gamma$-\textup(RIE), we first fix the suitable rough path lift. Note that $\bX^0$ corresponds to the It{\^o}-rough path lift and $\bX^{\frac{1}{2}}$ corresponds to the Stratonovich-rough path lift of a stochastic process, since the ``iterated integral'' $\X^0$ and $\X^{\frac{1}{2}}$ is given as a limit of left-point and mid-point Riemann sums, respectively, analogously to the stochastic It{\^o} and Stratonovich integral.

\begin{proposition}\label{prop: rough path lift under gamma RIE}
  Let $X$ and $\pi = (\pi^n)_{n\in \N}$ be as in Assumption~\ref{ass: partitions}. Suppose that $X \in C([0,T];\R^d)$ satisfies Property $\gamma$-\textup{(RIE)} relative to some $\gamma \in [0,1]$, $p \in (2,3)$ and a sequence $\pi = (\pi^n)_{n \in \N}$ of partitions. Then, $X$ extends canonically to a continuous $p$-rough path $\bX^\gamma := (X, \X^\gamma)$, where
  \begin{equation}\label{eq: rough path lift via gamma RIE}
    \X^\gamma_{s,t} := \int_0^t X_r \otimes \d^{\gamma,\pi} X_r - \int_0^s X_r \otimes \d^{\gamma,\pi} X_r - X_s \otimes X_{s,t}, \qquad (s,t) \in \Delta_T.
  \end{equation}
\end{proposition}
Note that the rough path lift $ \X^\gamma$ depends on the choice of the partition sequences $\pi$.
\begin{proof}[Proof of Proposition \ref{prop: rough path lift under gamma RIE}]
  It is straightforward to check that $(X,\X^\gamma)$ satisfies Chen's relation and that $\|X\|_p < \infty$. Therefore, it remains to show that $\|\X^\gamma\|_{\frac{p}{2}} < \infty$. We define $X^n \colon [0,T] \to \R^d$ by
  \begin{equation*}
    X^n_t = X_t \1_{\{T\}}(t) + \sum_{k=0}^{N_n-1} X_{t^n_k} \1_{[t^n_k,t^n_{k+1})}(t), \qquad t \in [0,T].
  \end{equation*}
  By Property $\gamma$-\textup{(RIE)}, we know that
  \begin{equation*}
    \lim_{n \to \infty} \X^{\gamma,\pi^n}_{s,t} = \X^\gamma_{s,t},
  \end{equation*}
  for $\X^{\gamma,\pi^n}_{s,t} := \int_0^t X_r \otimes \d^{\gamma,\pi^n} X_r - \int_0^s X_r \otimes \d^{\gamma,\pi^n} X_r - X^n_s \otimes X_{s,t}$, for $(s,t) \in \Delta_T$, where the convergence is uniform in $(s,t)$. We aim to show that $\sup_{n \in \N} \|\X^{\gamma,\pi^n}\|_{\frac{p}{2}} < \infty$, which then implies by the lower semi-continuity of the $\frac{p}{2}$-variation that
  \begin{equation*}
    \|\X^\gamma\|_{\frac{p}{2}} \leq \liminf_{n \to \infty} \|\X^{\gamma,\pi^n}\|_{\frac{p}{2}} < \infty.
  \end{equation*}
  Let $(s,t) \in \Delta_T$. If there exists $k$ such that $t^n_k \le s < t \le t^n_{k+1}$, then we estimate
  \begin{equation}\label{eq: gamma RIE pr1}
    |\X^{\gamma,\pi^n}_{s,t}|^{\frac{p}{2}} = |(X_{t^n_k} + \gamma X_{t^n_k,t^n_{k+1}}) \otimes X_{s,t} - X_{t^n_k} \otimes X_{s,t}|^{\frac{p}{2}} \lesssim |X_{s,t}|^p + |X_{t^n_k,t^n_{k+1}}|^p \lesssim c(t^n_k,t^n_{k+1}).
  \end{equation}
  Otherwise, let $k_0$ be the smallest $k$ such that $t^n_k \in (s,t)$, and let $k_1$ be the largest such $k$. We decompose
  \begin{equation*}
    \X^{\gamma,\pi^n}_{s,t} = \X^{\gamma,\pi^n}_{s,t^n_{k_0}} + \X^{\gamma,\pi^n}_{t^n_{k_0},t^n_{k_1}} + \X^{\gamma,\pi^n}_{t^n_{k_1},t} + X^n_{s,t^n_{k_0}} \otimes X_{t^n_{k_0}, t^n_{k_1}} + X^n_{s,t^n_{k_1}} \otimes X_{t^n_{k_1}, t}.
  \end{equation*}
  By~\eqref{eq: bounded p/2 variation assumption}, we have $|\X^{\gamma,\pi^n}_{t^n_{k_0},t^n_{k_1}}|^{\frac{p}{2}} \lesssim c(t^n_{k_0},t^n_{k_1})$, and we estimate
  \begin{align*}
    &|X^n_{s,t^n_{k_0}} \otimes X_{t^n_{k_0},t^n_{k_1}}|^{\frac{p}{2}} + |X^n_{s,t^n_{k_1}} \otimes X_{t^n_{k_1},t}|^{\frac{p}{2}} \\
    &\quad \lesssim |X^n_{s,t^n_{k_0}}|^p + |X_{t^n_{k_0},t^n_{k_1}}|^p + |X^n_{s,t^n_{k_1}}|^p + |X_{t^n_{k_1},t}|^p \\
    &\quad = |X_{t^n_{k_0-1},t^n_{k_0}}|^p + |X_{t^n_{k_0},t^n_{k_1}}|^p + |X_{t^n_{k_0-1},t^n_{k_1}}|^p + |X_{t^n_{k_1},t}|^p \\
    &\quad \lesssim 2 c(t^n_{k_0-1},t).
  \end{align*}
  Combining this with~\eqref{eq: gamma RIE pr1}, we deduce that $\|\X^{\gamma,\pi^n}\|_{\frac{p}{2}} \lesssim c(0,T)$, and the proof is complete.
\end{proof}

We now proceed similarly to~\cite{Perkowski2016}. The following lemma links Property $\gamma$-\textup{(RIE)} to the existence of quadratic variation, which we rely on when calculating the rough integral.

\begin{lemma}\label{lemma: quadratic variation under gamma RIE}
  Let $X$ and $\pi = (\pi^n)_{n\in \N}$ as in Assumption~\ref{ass: partitions}. Suppose that $X \in C([0,T];\R^d)$ satisfies Property $\gamma$-\textup{(RIE)} relative to some $\gamma \in [0,1]$, $p \in (2,3)$ and a sequence $\pi = (\pi^n)_{n \in \N}$ of partitions. Let $1 \leq i, j \leq d$, and define for $\gamma = \frac{1}{2}$, $[X^i,X^j]^{\gamma,\pi} := 0$, and for $\gamma \neq \frac{1}{2}$,
  \begin{equation*}
    [X^i,X^j]^{\gamma,\pi}_t := X^i_t X^j_t - X^i_0 X^j_0 - \int_0^t X^i_s \dd^{\gamma,\pi} X^j_s - \int_0^t X^j_s \dd^{\gamma,\pi} X^i_s, \qquad t \in [0,T].
  \end{equation*}
  Then, $[X^i,X^j]^{\gamma,\pi}$ is a continuous function and
  \begin{equation}\label{eq: quadratic variation as limit under gamma RIE}
    [X^i,X^j]^{\gamma,\pi}_t = \lim_{n \to \infty} [X^i,X^j]^{\gamma,\pi^n}_t := \lim_{n \to \infty} (1-2\gamma) \sum_{k=0}^{N_n-1} X^i_{t^n_k \wedge t, t^n_{k+1} \wedge t} X^j_{t^n_k \wedge t, t^n_{k+1} \wedge t},\quad t \in [0,T].
  \end{equation}
  The sequence $([X^i,X^j]^{\gamma,\pi^n})_{n \in \N}$ has uniformly bounded $1$-variation and, thus, $[X^i,X^j]^{\gamma,\pi}$ has finite $1$-variation. We write $[X]^{\gamma,\pi} = [X,X]^{\gamma,\pi} = ([X^i,X^j]^{\gamma,\pi})_{1 \leq i,j \leq d}$, and, analogously, $[X]^{\gamma,\pi^n}$, $n \in \N$.
\end{lemma}

\begin{proof}
  By definition, the function $t\mapsto [X^i,X^j]_{ t}^{\gamma,\pi}$ is continuous. We observe that
  \begin{equation*}
    X^i_t X^j_t - X^i_0 X^j_0 = \sum_{k=0}^{N_n-1} (X^i_{t^n_{k+1} \wedge t} X^j_{t^n_{k+1} \wedge t} - X^i_{t^n_k \wedge t} X^j_{t^n_k \wedge t})
  \end{equation*}
  for every $n \in \N$, and
  \begin{align*}
    &X^i_{t^n_{k+1} \wedge t} X^j_{t^n_{k+1} \wedge t} - X^i_{t^n_k \wedge t} X^j_{t^n_k \wedge t} \\
    &\quad = (X^i_{t^n_k \wedge t} + \gamma X^i_{t^n_k \wedge t, t^n_{k+1} \wedge t}) X^j_{t^n_k \wedge t, t^n_{k+1} \wedge t} + (X^j_{t^n_k \wedge t} + \gamma X^j_{t^n_k \wedge t, t^n_{k+1} \wedge t})  X^i_{t^n_k \wedge t, t^n_{k+1} \wedge t} \\
    &\qquad + (1 - 2 \gamma) X^i_{t^n_k \wedge t, t^n_{k+1} \wedge t} X^j_{t^n_k \wedge t, t^n_{k+1} \wedge t}.
  \end{align*}
  Since $(\int_0^\cdot X_s \otimes \d^{\gamma,\pi^n} X_s)_{n\in \N}$ converges uniformly to $\int_0^\cdot X_s \otimes \d^{\gamma,\pi} X_s$, the convergence in~\eqref{eq: quadratic variation as limit under gamma RIE}  then holds. We further see that
  \begin{equation*}
    X^i_{t^n_k \wedge t, t^n_{k+1} \wedge t} X^j_{t^n_k \wedge t, t^n_{k+1} \wedge t} = \frac{1}{4} (((X^i + X^j)_{t^n_k \wedge t, t^n_{k+1} \wedge t})^2 - ((X^i - X^j)_{t^n_k \wedge t, t^n_{k+1} \wedge t})^2)
  \end{equation*}
  (i.e.~$[X^i,X^j]^{\gamma,\pi} = \frac{1}{4} ([X^i + X^j]^{\gamma,\pi} - [X^i - X^j]^{\gamma,\pi})$). That is, $[X^i,X^j]^{\gamma,\pi^n}$ is given as the difference of two increasing functions, and its $1$-variation is bounded from above by
  \begin{align*}
    &(1-2\gamma) \sum_{k=0}^{N_n-1} (((X^i + X^j)_{t^n_k,t^n_{k+1}})^2 + ((X^i - X^j)_{t^n_k, t^n_{k+1}})^2) \\
    &\quad \lesssim (1-2\gamma) \sup_{m \in \N} \sum_{k=0}^{N_m-1} ((X^i_{t^m_k, t^m_{k+1}})^2 + (X^j_{t^m_k, t^m_{k+1}})^2).
  \end{align*}
  Since the right-hand side is finite, we obtain that the limit $[X^i,X^j]^{\gamma,\pi}$ has finite $1$-variation.
\end{proof}

With the quadratic variation at hand, we apply a piecewise linear interpolation to continuously approximate the path and obtain a Stratonovich-type integral, that we then translate back into a general pathwise integral.

\begin{lemma}\label{lemma: gamma RIE}
  Let $X$ and $\pi = (\pi^n)_{n\in \N}$ be as in Assumption~\ref{ass: partitions}. Suppose that $X$ satisfies Property $\gamma$-\textup{(RIE)} relative to some $\gamma \in [0,1]$, $p \in (2,3)$ and a sequence of partitions $\pi = (\pi^n)_{n \in \N}$. Define $(\bar{X}^n)_{n\in \N}$ as the piecewise linear interpolation of $X$ along $\pi = (\pi^n)_{n \in \N}$. Then the Riemann--Stieltjes integral
  \begin{equation}\label{eq: Stratonovich vs gamma RIE}
    \begin{split}
    \lim_{n \to \infty} \int_s^t \bar{X}^n_r \otimes \d \bar{X}^n_r 
    &:= \lim_{n \to \infty} \sum_{k=0}^{N_n-1} (X_{t^n_k} + \frac{1}{2} X_{t^n_k,t^n_{k+1}}) \otimes X_{t^n_k \wedge t, t^n_{k+1} \wedge t} \\
    &= \int_s^t X_r \otimes \d^{\gamma,\pi} X_r + \frac{1}{2} [X]^{\gamma,\pi}_{s,t},
    \end{split}
  \end{equation}
  where the convergence is uniform in $(s,t) \in \Delta_T$. Moreover, the sequence $(\bar{\X}^n)_{n \in \N}$ has uniformly bounded $\frac{p}{2}$-variation, where $\bar{\X}^n_{s,t} := \int_s^t \bar{X}^n_{s,r} \otimes \d \bar{X}^n_r$ for $(s,t) \in \Delta_T$.
\end{lemma}

\begin{proof}
  Let $n \in \N$ and $0 \leq k \leq N_{n-1}$. By definition, for $t \in [t^n_k,t^n_{k+1}]$, we have
  \begin{equation*}
    \bar{X}^n_t = X_{t^n_k} + \frac{t-t^n_k}{t^n_{k+1}-t^n_k} X_{t^n_k,t^n_{k+1}},
  \end{equation*}
  which gives that
  \begin{equation}\label{eq: stratonovich pr1}
    \begin{split}
    &\int_{t^n_k}^{t^n_{k+1}} \bar{X}^n_r \otimes \d \bar{X}^n_r = (X_{t^n_k} + \frac{1}{2} X_{t^n_k,t^n_{k+1}}) \otimes X_{t^n_k,t^n_{k+1}} \\
    &\quad = (X_{t^n_k} + \gamma X_{t^n_k,t^n_{k+1}}) \otimes X_{t^n_k,t^n_{k+1}} + \frac{1}{2}(1 - 2 \gamma) X_{t^n_k,t^n_{k+1}} \otimes X_{t^n_k,t^n_{k+1}}.
    \end{split}
  \end{equation}
  Lemma~\ref{lemma: quadratic variation under gamma RIE} then implies the uniform convergence and~\eqref{eq: Stratonovich vs gamma RIE}.

  We now show that $(\bar{\X}^n)_{n \in \N}$ has uniformly bounded $\frac{p}{2}$-variation. Let $(s,t) \in \Delta_T$. If $t^n_k \le s < t \le t^n_{k+1}$ for some $k$, then we estimate
  \begin{equation}\label{eq:stratonovich pr2}
    \begin{split}
    |\bar{\X}^n_{s,t}|^{\frac{p}{2}} &= \Big| \int_s^t \bar{X}^n_{s,r} \otimes \d \bar{X}^n_r \Big|^{\frac{p}{2}} \le \Big| \int_s^t (r-s) \frac{|X_{t^n_k,t^n_{k+1}}|^2}{|t^n_{k+1} - t^n_k|^2} \dd r \Big|^{\frac{p}{2}} \\
    & = \frac{1}{2^{\frac{p}{2}}} |t-s|^p \frac{|X_{t^n_k,t^n_{k+1}}|^p}{|t^n_{k+1} - t^n_k|^p} \le \frac{|t-s|}{|t^n_{k+1} - t^n_k|} \|X\|_{p,[t^n_k,t^n_{k+1}]}^p.
    \end{split}
  \end{equation}
  Otherwise, let $k_0$ be the smallest $k$ such that $t^n_k \in (s,t)$, and let $k_1$ be the largest such $k$. It is straightforward to see that $(\bar{X}^n,\bar{\X}^n)$ satisfies Chen's relation:
  \begin{equation*}
    \bar{\X}^n_{s,t} = \bar{\X}^n_{s,u} + \bar{\X}^n_{u,t} + \bar{X}^n_{s,u} \otimes \bar{X}^n_{u,t}
  \end{equation*}
  for all $s \leq u \leq t$, from which it follows that
  \begin{equation*}
    \bar{\X}^n_{s,t} = \bar{\X}^n_{s,t^n_{k_0}} + \bar{\X}^n_{t^n_{k_0},t^n_{k_1}} + \bar{\X}^n_{t^n_{k_1},t} + \bar{X}^n_{s,t^n_{k_0}} \otimes \bar{X}^n_{t^n_{k_0}, t^n_{k_1}} + \bar{X}^n_{s,t^n_{k_1}} \otimes \bar{X}^n_{t^n_{k_1}, t}.
  \end{equation*}
  Recalling the calculation~\eqref{eq: stratonovich pr1}, we get that
  \begin{equation*}
    |\bar{\X}^n_{t^n_{k_0},t^n_{k_1}}|^{\frac{p}{2}} \lesssim \Big| \Big( \int_0^\cdot X_s \otimes \d^{\gamma,\pi^n} X_s \Big)_{t^n_{k_0},t^n_{k_1}} - X_{t^n_{k_0}} \otimes X_{t^n_{k_0},t^n_{k_1}} \Big|^{\frac{p}{2}} + |[X]^{\gamma,\pi^n}_{t^n_{k_0},t^n_{k_1}}|^{\frac{p}{2}},
  \end{equation*}
  where $[X]^{\gamma,\pi^n}$ is defined in Lemma~\ref{lemma: quadratic variation under gamma RIE}. Using the inequality in~\eqref{eq: bounded p/2 variation assumption} and Lemma~\ref{lemma: quadratic variation under gamma RIE}, we see that there exists a control function $\bar{c}$ such that the right-hand side is bounded from above by $\bar{c}(t^n_{k_0},t^n_{k_1})$. If we combine this with the estimate~\eqref{eq:stratonovich pr2} and a simple estimate for the terms $\bar{X}^n_{s,t^n_{k_0}} \otimes \bar{X}^n_{t^n_{k_0}, t^n_{k_1}}$ and $\bar{X}^n_{s,t^n_{k_1}} \otimes \bar{X}^n_{t^n_{k_1}, t}$, we can conclude that $\| \bar{\X}^n \|_{\frac{p}{2}} \lesssim \bar{c}(0,T) + \|X\|_p^2$, which completes the proof.
\end{proof}

Having made these preparations, we are now able to prove that the rough integral can be obtained as a limit of (non-compensated) general Riemann sums, given that the driving path satisfies Property $\gamma$-\textup{(RIE)}.

\begin{theorem}\label{thm: general pathwise stochastic integral}
  Let $X$ and $\pi = (\pi^n)_{n\in \N}$ be as in Assumption~\ref{ass: partitions}. Suppose that $X$ satisfies Property $\gamma$-\textup{(RIE)} relative to some $\gamma \in [0,1]$, $p \in (2,3)$ and a sequence of partitions $\pi = (\pi^n)_{n \in \N}$. Let $q > 0$ be such that $\frac{2}{p} + \frac{1}{q} > 1$ and let $(Y,Y') \in \mathscr{C}^{p,q}_X([0,T];\R^{m \times d})$ be a controlled path. Then, the rough integral $\int Y \d \bX^\gamma$ satisfies
  \begin{equation*}
    \int_0^t Y_s \dd \bX_s^\gamma 
    =\int_0^t Y_s \dd^{\gamma,\pi} X_s
    := \lim_{n\to \infty} \sum_{k=0}^{N_n-1} (Y_{t^n_k} +\gamma Y_{t^n_k,t^n_{k+1}}) X_{t^n_k \wedge t, t^n_{k+1} \wedge t},
  \end{equation*}
  where the convergence is uniform in $t \in [0,T]$.
\end{theorem}

\begin{proof}
  We denote by $(\bar{X}^n)_{n\in \N}$ and $(\bar{Y}^n)_{n\in \N}$ the piecewise linear interpolation of $X$ and $Y$, respectively, along $\pi = (\pi^n)_{n \in \N}$. Thus, $(\bar{Y}^n,Y')$ is controlled by $\bar{X}^n$, with remainder $R^{\bar{Y}^n}_{s,t} = \bar{Y}^n_{s,t} - Y'_s \bar{X}^n_{s,t}$, $(s,t) \in \Delta_T$. As shown in the proof of~\cite[Theorem~4.19]{Perkowski2016}, if $p' > p$ and $q' > q$ such that $\frac{2}{p'} + \frac{1}{q'} > 1$, then $(\bar{Y}^n,Y',R^{\bar{Y}^n})$ converges in $(q',p',r')$-variation to $(Y,Y',R^Y)$, where $\frac{1}{r'} = \frac{1}{p'} + \frac{1}{q'}$.

  Since the sequence $(\bar{X}^n)_{n \in \N}$ has uniformly bounded $p$-variation and $\bar{X}^n$ converges uniformly to $X$ as $n \to \infty$, it follows by interpolation that $\bar{X}^n$ converges to $X$ with respect to the $p'$-variation norm, i.e.~$\|\bar{X}^n - X\|_{p'} \to 0$ as $n \to \infty$. It follows similarly using Lemma~\ref{lemma: gamma RIE} that $\|(\bar{\X}^n - (\X^\gamma + \frac{1}{2} [X]^{\gamma,\pi})\|_{\frac{p'}{2}} \to 0$ and, hence, that $\|(\bar{X}^n,\bar{\X}^n) - (X,\X^\gamma + \frac{1}{2} [X]^{\gamma,\pi})\|_{p'} \to 0$ as $n \to \infty$.

  The continuity of the It{\^o}--Lyons map, see e.g. \cite[Theorem~4.17]{Friz2020}, yields the uniform convergence of the rough integrals $\int \bar{Y}^n \dd (\bar{X}^n,\bar{\X}^n)$ to the rough integral $\int Y \dd (X, \X^\gamma + \frac{1}{2} [X]^{\gamma,\pi})$. But, for every $t \in [0,T]$, it holds that
  \begin{align*}
    &\lim_{n \to \infty} \int_0^t \bar{Y}^n_s \dd (\bar{X}^n,\bar{\X}^n)_s \\
    &\quad = \lim_{n \to \infty} \int_0^t \bar{Y}^n_s \dd \bar{X}^n_s \\
    &\quad = \lim_{n \to \infty} \sum_{k=0}^{N_n-1} (Y_{t^n_k} + \frac{1}{2} Y_{t^n_k,t^n_{k+1}}) X_{t^n_k \wedge t, t^n_{k+1} \wedge t} \\
    &\quad = \lim_{n \to \infty} \bigg(\sum_{k=0}^{N_n-1} (Y_{t^n_k} + \gamma Y_{t^n_k,t^n_{k+1}}) X_{t^n_k \wedge t, t^n_{k+1} \wedge t} + \frac{1}{2}(1 - 2 \gamma) \sum_{k=0}^{N_n-1} Y_{t^n_k,t^n_{k+1}} X_{t^n_k \wedge t, t^n_{k+1} \wedge t} \bigg).
  \end{align*}
  Since $(Y,Y') \in \mathscr{C}^{p,q}_X$, it is immediate that the second term on the right-hand side converges uniformly to $\frac{1}{2} \int_0^t Y'_s \dd [X]^{\gamma,\pi}_s$, $t \in [0,T]$. Thus, we have that
  \begin{align*}
    &\lim_{n \to \infty} \sum_{k=0}^{N_n-1} (Y_{t^n_k} + \gamma Y_{t^n_k,t^n_{k+1}}) X_{t^n_k \wedge t, t^n_{k+1} \wedge t} \\
    &\quad = \lim_{n \to \infty} \int_0^t \bar{Y}^n_s \dd (\bar{X}^n,\bar{\X}^n)_s - \frac{1}{2} \int_0^t Y'_s \dd [X]^{\gamma,\pi}_s \\
    &\quad = \int_0^t Y_s \dd (X, \X^{\gamma,\pi} + \frac{1}{2} [X]^{\gamma,\pi})_s - \frac{1}{2} \int_0^t Y'_s \dd [X]^{\gamma,\pi}_s \\
    &\quad = \lim_{|\mathcal{P}| \to 0} \sum_{[u,v] \in \mathcal{P}} Y_u  X_{u,v} + Y'_u (\X^\gamma + \frac{1}{2} [X]^{\gamma,\pi})_{u,v} -  \frac{1}{2} \lim_{|\mathcal{P}| \to 0} \sum_{[u,v] \in \mathcal{P}} Y'_u  [X]^{\gamma,\pi}_{u,v} \\
    &\quad = \lim_{|\mathcal{P}| \to 0} \sum_{[u,v] \in \mathcal{P}} Y_u X_{u,v} + Y'_u \X^\gamma_{u,v} \\  &\quad = \int_0^t Y_s \dd \bX^\gamma_s,
  \end{align*}
  where the limit is taken over any sequence of partitions $\mathcal P$ of $[0,t]$ with vanishing mesh. 
\end{proof}

\begin{remark}
  Theorem~\ref{thm: general pathwise stochastic integral} is a generalization of \cite[Theorem~4.19]{Perkowski2016} which states that the rough integral can be obtained as a limit of left-point Riemann sums, given that the driving path satisfies Property $\gamma$-\textup{(RIE)} for $\gamma = 0$.
\end{remark}

\begin{remark}
The rough path lift $\mathbb X^2$ defined in \cite[Lemma~4.7]{Cont2019} coincides with the symmetric part of the rough path lift $\X^\gamma$ for $\gamma = 0$ defined in Proposition~\ref{prop: rough path lift under gamma RIE} (given that $X$ satisfies Property \textup{(RIE)} relative to the dyadic Lebesgue partitions). The corresponding rough path then is a so-called reduced rough path, whereas $\bX^\gamma$ defines a classical rough path in the sense that the area (the antisymmetric part) is also included.
\end{remark}

While we discuss several examples of paths~$X$, generated by stochastic processes, satisfying the assumptions of Theorem~\ref{thm: general pathwise stochastic integral} in Section~\ref{sec: application to stochastic integration} below, let us briefly provide some standard examples of controlled paths to demonstrate the scope of Theorem~\ref{thm: general pathwise stochastic integral}.

\begin{remark}
  Let $X \in C^{p\text{-var}}([0,T];\R^d)$ be a path of finite $p$-variation for $p\in (2,3)$.
  \begin{enumerate}
    \item[(i)] Any path $Y \in C^{q\text{-var}}([0,T];\R^{m \times d})$ for $q \in (1,2)$ such that $\frac{1}{p} + \frac{1}{q} > 1$ is a controlled path w.r.t.~$X$. In this case we can choose $Y^\prime = 0$ and the rough integral coincides with the Young integral.
    \item[(ii)] For any function $f \colon \R^d \to \R^{m \times d}$ which is twice continuously differentiable, $f(X)$ is a controlled path w.r.t.~$X$, see e.g. \cite[Example~4.7]{Perkowski2016}. 
    \item[(iii)] For a non-anticipative, path-dependent functional $F\colon [0,T]\times C([0,T];\R^d)\to \R^{m\times d}$ with suitable differentiability properties in the sense of Dupire~\cite{Dupire2019}, $F(\cdot,X)$ is controlled paths w.r.t.~$X$, see \cite[Section~4]{Ananova2023}.
    \item[(iv)] Any solution $Y$ of a rough differential equation, like
    \begin{equation*}
      Y_t = y_0 + \int_0^t F(s,Y_s) \dd \bX_s, \quad t\in [0,T],
    \end{equation*}
    is a controlled path w.r.t.~$X$, where we assume that $\bX = (X,\X)$ is a rough path and $F$ is suitable function, see e.g. \cite[Section~3]{Friz2018}.
  \end{enumerate}
\end{remark}

\subsection{On the convergence of general Riemann integrals}\label{sec: on the assumption of general Riemann integrals}

In this subsection we establish the relation of Property $\gamma$-\textup{(RIE)} to the existence of quadratic variation and L{\'e}vy area as well as to Property \textup{(RIE)} from \cite{Perkowski2016}, which reads as follows for a continuous path.

\begin{RIE}
  Let $X$ and $\pi = (\pi^n)_{n\in \N}$ be as in Assumption~\ref{ass: partitions}. We assume that
  \begin{enumerate}
    \item[(i)] the left-point Riemann sums
    \begin{equation*}
      \int_0^t X_s \otimes \d^{\pi^n} X_s := \sum_{k=0}^{N_n-1} X_{t^n_k} \otimes X_{t^n_k \wedge t, t^n_{k+1} \wedge t}, \quad t \in [0,T],
    \end{equation*}
    converge uniformly as $n \to \infty$ to a limit, which we denote by $\int_0^t X_s \otimes \d^\pi X_s$
    \item[(ii)] there exists a constant $p \in (2,3)$ and a control function $c$ such that
    \begin{equation*}
      \sup_{(s,t) \in \Delta_T} \frac{|X_{s,t}|^p}{c(s,t)} + \sup_{n \in \N} \sup_{0 \le k < \ell \le N_n} \frac{|(\int_0^{\cdot} X_s \otimes \d^{\pi^n} X_s)_{t^n_k,t^n_\ell} - X_{t^n_k} \otimes X_{t^n_k,t^n_\ell}|^{\frac{p}{2}}}{c(t^n_k,t^n_\ell)} \lesssim 1.
    \end{equation*}
  \end{enumerate}
\end{RIE}

We say that a path $X \in C([0,T];\R^d)$ satisfies \emph{Property \textup{(RIE)}} relative to $p$ and $\pi$ if $p$, $\pi$ and $X$ together satisfy Property \textup{(RIE)}.

\medskip

Let us begin by briefly summarizing the findings, which we shall prove below.

\smallskip

\noindent Case~1: For $\gamma \neq \frac{1}{2}$, we obtain:
\begin{equation*}
  \text{Property \textup{(RIE)} (i)}
  \hspace{9pt} \overset{\text{Lemma~\ref{lem: RIE eq to gamma RIE}}}{\Longleftrightarrow}\hspace{9pt} \text{Property $\gamma$-\textup(RIE) (i)}
  \hspace{9pt} \overset{\text{Lemma~\ref{lem: RIE eq to Levy and quad var}}}{\Longleftrightarrow}\hspace{9pt}
  \begin{cases}
  \text{Quadratic variation}\\
  \text{L{\'e}vy area}
  \end{cases}
\end{equation*}
and
\begin{equation*}
  \text{Property \textup{(RIE)} (ii)}
  \quad \overset{\text{Lemma~\ref{lem: RIE eq to gamma RIE}}}{\Longleftrightarrow} \quad \text{Property $\gamma$-\textup(RIE) (ii)}
  \quad \overset{\text{Lemma~\ref{lem: RIE eq to Levy and quad var}}}{\Longleftrightarrow} \quad
  \text{Regularity cond.~\ref{ass: regularity}}
\end{equation*}
\noindent Case~2: For $\gamma = \frac{1}{2}$, we obtain:
\begin{equation*}
  \text{Property \textup{(RIE)} (i)}
  \quad \overset{\text{Lemma~\ref{lem: RIE eq to gamma RIE}}}{\Longrightarrow}\quad \text{Property $\gamma$-\textup(RIE) (i)}
  \quad \overset{\text{Lemma~\ref{lem: Levy eq to 1/2 RIE}}}{\Longleftrightarrow}\quad
  \text{L{\'e}vy area}
\end{equation*}
and
\begin{equation*}
  \text{Property \textup{(RIE)} (ii)}
  \quad \overset{\text{Lemma~\ref{lem: RIE eq to gamma RIE}}}{\Longrightarrow}\quad \text{Property $\gamma$-\textup(RIE) (ii)}
  \quad \overset{\text{Lemma~\ref{lem: Levy eq to 1/2 RIE}}}{\Longleftrightarrow}\quad
  \text{Regularity cond.~\ref{ass: regularity}}
\end{equation*}

In the next lemma, we derive the relation between Property $\gamma$-\textup{(RIE)} and Property \textup{(RIE)}.

\begin{lemma}\label{lem: RIE eq to gamma RIE}
  Let $\gamma \in [0,1]$, $p \in (2,3)$, $X$ and $\pi = (\pi^n)_{n\in \N}$ be as in Assumption~\ref{ass: partitions}.
  \begin{enumerate}
    \item[(i)] Suppose $\gamma \neq \frac{1}{2}$. Then, $X$ satisfies Property \textup{(RIE)} (i) relative to $\pi$ if and only if $X$ satisfies Property $\gamma$-\textup{(RIE)} (i) relative to $\gamma$ and $\pi$, and $X$ satisfies Property \textup{(RIE)} (ii) relative to $p$ and $\pi$ if and only if $X$ satisfies Property $\gamma$-\textup{(RIE)} (ii) relative to $\gamma$, $p$ and $\pi$.
    \item[(ii)] Suppose $\gamma = \frac{1}{2}$. If $X$ satisfies Property \textup{(RIE)} (i) relative to $\pi$, then $X$ satisfies Property $\gamma$-{\textup{(RIE)}} (i) relative to $\gamma$ and $\pi$, and if $X$ satisfies Property \textup{(RIE)} (ii) relative to $p$ and $\pi$, then $X$ satisfies Property $\gamma$-\textup{(RIE)} (ii) relative to $\gamma$, $p$ and $\pi$.
  \end{enumerate}
\end{lemma}

\begin{proof}
  First, note that
  \begin{equation}\label{eq:gamma integral in Riemann integral}
    \begin{split}
    &\int_0^t X_s \otimes \d^{\gamma,\pi^n} X_s
    := \sum_{k=0}^{N_n-1} (X_{t^n_k} + \gamma X_{t^n_k,t^n_{k+1}}) \otimes X_{t^n_k \wedge t, t^n_{k+1} \wedge t} \\
    &\quad = \sum_{k=0}^{N_n-1} X_{t^n_k} \otimes X_{t^n_k \wedge t, t^n_{k+1} \wedge t} + \gamma \sum_{k=0}^{N_n-1} X_{t^n_k \wedge t, t^n_{k+1} \wedge t} \otimes X_{t^n_k \wedge t, t^n_{k+1} \wedge t} \\
    &\qquad + \gamma \sum_{k=0}^{N_n-1} (X_{t^n_{k+1} \wedge t, t^n_{k+1}} - X_{t^n_k \wedge t, t^n_k}) \otimes X_{t^n_k \wedge t, t^n_{k+1} \wedge t} \\
    &\quad = \int_0^t X_s \otimes \d^{\pi^n} X_s
    + \gamma [X]^{\pi^n}_t + \gamma \sum_{k=0}^{N_n-1} (X_{t^n_{k+1} \wedge t, t^n_{k+1}} - X_{t^n_k \wedge t, t^n_k}) \otimes X_{t^n_k \wedge t, t^n_{k+1} \wedge t}
    \end{split}
  \end{equation}
  for $t\in [0,T]$, where we write $[X]^{\pi^n} := \sum_{k=1}^{N_n-1} X_{t^n_k \wedge \cdot, t^n_{k+1} \wedge \cdot} \otimes X_{t^n_k \wedge \cdot, t^n_{k+1} \wedge \cdot}$ and note that, for $\gamma \neq \frac{1}{2}$, $[X]^{\pi^n} = \frac{1}{1-2\gamma} [X]^{\gamma,\pi^n}$. Secondly, note that, for any control function $c$, we have
  \begin{equation}\label{eq: bound 1}
    \begin{split}
    &\sup_{0 \le k < \ell \le N_n} \frac{| (\int_0^{\cdot} X_s \otimes \d^{\gamma,\pi^n} X_s)_{t^n_k,t^n_\ell} - (X_{t^n_k} + \gamma X_{t^n_k,t^n_\ell}) \otimes X_{t^n_k,t^n_\ell}|^{\frac{p}{2}}}{c(t^n_k,t^n_\ell)}\\
    &\quad\leq 3^{\frac{p}{2}-1} \sup_{0 \le k < \ell \le N_n} \frac{|(\int_0^{\cdot} X_s \otimes \d^{\pi^n} X_s)_{t^n_k,t^n_\ell} - X_{t^n_k} \otimes X_{t^n_k,t^n_\ell}|^{\frac{p}{2}}}{c(t^n_k,t^n_\ell)} + 3^{\frac{p}{2}-1} \gamma^{\frac{p}{2}} \sup_{0 \le k < \ell \le N_n}  \frac{|[X]^{\pi^n}_{t^n_k,t^n_\ell}|^{\frac{p}{2}}}{c(t^n_k,t^n_\ell)} \\
    &\qquad + 3^{\frac{p}{2}-1} \gamma^{\frac{p}{2}} \frac{| X_{t^n_k,t^n_\ell} \otimes X_{t^n_k,t^n_\ell}|^{\frac{p}{2}}}{c(t^n_k,t^n_\ell)}
    \end{split}
  \end{equation}
  and, for $\gamma \neq \frac{1}{2}$, we get
  \begin{equation}\label{eq: bound 2}
    \begin{split}
    &\sup_{0 \le k < \ell \le N_n} \frac{| (\int_0^{\cdot} X_s \otimes \d^{\pi^n} X_s)_{t^n_k,t^n_\ell} - X_{t^n_k} \otimes X_{t^n_k,t^n_\ell}|^{\frac{p}{2}}}{c(t^n_k,t^n_\ell)}\\
    &\quad \leq
    3^{\frac{p}{2}-1} \sup_{0 \le k < \ell \le N_n} \frac{|(\int_0^{\cdot} X_s \otimes \d^{\gamma,\pi^n} X_s)_{t^n_k,t^n_\ell} - (X_{t^n_k} + \gamma X_{t^n_k,t^n_\ell}) \otimes X_{t^n_k,t^n_\ell}|^{\frac{p}{2}}}{c(t^n_k,t^n_\ell)}  \\
    &\qquad + \frac{3^{\frac{p}{2}-1} \gamma^{\frac{p}{2}} }{|1-2\gamma|^{\frac{p}{2}}} \sup_{0 \le k < \ell \le N_n}  \frac{|[X]^{\gamma,\pi^n}_{t^n_k,t^n_\ell} |^{\frac{p}{2}}}{c(t^n_k,t^n_\ell)} + 3^{\frac{p}{2}-1} \gamma^{\frac{p}{2}} \frac{|X_{t^n_k,t^n_\ell} \otimes X_{t^n_k,t^n_\ell}|^{\frac{p}{2}}}{c(t^n_k,t^n_\ell)}.
    \end{split}
  \end{equation}
  If $X$ satisfies Property \textup{(RIE)} (i), $(\int_0^{\cdot} X_s \otimes \d^{\pi^n} X_s)_{n\in \N}$ converges uniformly to $(\int_0^{\cdot} X_s \otimes \d^\pi X_s)$ and, by \cite[Lemma~4.17]{Perkowski2016}, $([X]^{\pi^n})_{n\in \N}$ converges uniformly to $[X]^\pi$ as $n \to \infty$. Moreover, again due to \cite[Lemma~4.17]{Perkowski2016}, $([X]^{\pi^n})_{n \in \N}$ has uniformly bounded $1$-variation. Hence, by \eqref{eq:gamma integral in Riemann integral} and \eqref{eq: bound 1}, Property \textup{(RIE)} (i) and (ii) imply Property $\gamma$-\textup{(RIE)} (i) and (ii), respectively, for every $\gamma \in [0,1]$.

  Conversely, if $\gamma \neq \frac{1}{2}$, using Lemma~\ref{lemma: quadratic variation under gamma RIE}, \eqref{eq:gamma integral in Riemann integral} and \eqref{eq: bound 2} yields that Property $\gamma$-\textup{(RIE)} (i) and (ii) imply Property \textup{(RIE)} (i) and (ii), respectively.
\end{proof}

To explain how Property \textup{(RIE)} and Property $\gamma$-\textup{(RIE)} are related to the existence of quadratic variation and L{\'e}vy area, we make the following observation.

\begin{remark}\label{rem:rough path and levy area}
  Let $\gamma \in [0,1]$ and let $X$ and $\pi = (\pi^n)_{n\in \N}$ be as in Assumption~\ref{ass: partitions}. Assuming the respective limits exist, we write
  \begin{equation*}
    (\X^{\gamma,\pi}_{s,t})^{ij} = \int_0^t X^i_r \dd^{\gamma,\pi} X^j_r - \int_0^s X^i_r \dd^{\gamma,\pi} X^j_r - X^i_s X^j_{s,t}, \quad(s,t) \in \Delta_T.
  \end{equation*}
  We decompose the iterated integrals into the symmetric and antisymmetric components as follows:
  \begin{align*}
    &(\X^{\gamma,\pi}_{s,t})^{ij} = \frac{1}{2} ((\X^{\gamma,\pi}_{s,t})^{ij} + (\X^{\gamma,\pi}_{s,t})^{ij}) + \frac{1}{2} ((\X^{\gamma,\pi}_{s,t})^{ij} - (\X^{\gamma,\pi}_{s,t})^{ij}) \\
    &\quad = \frac{1}{2} (X^i_{s,t} X^j_{s,t} - [X^i,X^j]^{\gamma,\pi}_{s,t}) + \frac{1}{2} (\cL^\pi_{s,t}(X)^{ij} - (X_s^i + X_t^i) X_{s,t}^j) \\
    &\quad =: \frac{1}{2} (\mathbb{S}(X)_{s,t}^{\gamma,\pi})^{ij} + \frac{1}{2} (\mathbb{A}(X)_{s,t}^{\pi})^{ij},
  \end{align*}
  for every $i, j = 1, \ldots, d$.

  For $\gamma = \frac{1}{2}$, we notice that the symmetric part reduces to $\frac{1}{2} X_{s,t} \otimes X_{s,t}$. We realize that for the Stratonovich-type rough path lift (implying the Stratonovich-type integral) to be well-posed in the rough path sense, it is only required that the antisymmetric Riemann sums converge (which do not depend on $\gamma$, see Remark~\ref{rem: Levy area}), and that the approximative L{\'e}vy area has uniformly bounded $\frac{p}{2}$-variation and the path has finite $p$-variation. For the more general case, it is additionally required that the symmetric Riemann sums converge. This suffices since the approximative quadratic variation term has uniformly bounded $1$-, thus $\frac{p}{2}$-variation by definition.
\end{remark}

Keeping Remark~\ref{rem:rough path and levy area} in mind, it is not surprising that $X$ satisfying Property ($\gamma$-)\textup{(RIE)} is equivalent to $X$ possessing L{\'e}vy area and, if $\gamma \neq \frac{1}{2}$, quadratic variation, together with Assumption~\ref{ass: regularity}. This is the content of Lemma~\ref{lem: RIE eq to Levy and quad var}, Corollary~\ref{cor: gamma RIE eq to Levy and quad var} and Lemma~\ref{lem: Levy eq to 1/2 RIE} below.

\begin{lemma}\label{lem: RIE eq to Levy and quad var}
  Let $X$ and $\pi = (\pi^n)_{n\in \N}$ be as in Assumption~\ref{ass: partitions}. Then, $X$ satisfies Property \textup{(RIE)} (i) relative to $\pi$ if and only if $X$ possesses quadratic variation and L{\'e}vy area, both along $\pi$, and $X$ satisfies Property \textup{(RIE)} (ii) relative to $p$ and $\pi$ if and only if $X$ satisfies Assumption~\ref{ass: regularity} relative to $p$ and $\pi$.
\end{lemma}

\begin{proof}
  First, note that
  \begin{equation}\label{eq: Levy area in Riemann integral plus quadratic variation}
    \begin{split}
    &\cL^{\pi^n}_t(X) = \sum_{k=0}^{N_n-1} (X_{t^n_k \wedge t} + X_{t^n_{k+1} \wedge t}) \otimes X_{t^n_k \wedge t, t^n_{k+1} \wedge t} \\
    &\quad = 2 \sum_{k=0}^{N_n-1} X_{t^n_k \wedge t} \otimes X_{t^n_k \wedge t, t^n_{k+1} \wedge t} + \sum_{k=0}^{N_n-1} X_{t^n_k \wedge t, t^n_{k+1} \wedge t} \otimes X_{t^n_k \wedge t, t^n_{k+1} \wedge t} \\
    &\quad =  2 \int_0^t X_s \otimes \d^{\pi^n} X_s + 2 \sum_{k=0}^{N_n-1} X_{t^n_k, t^n_k \wedge t} \otimes X_{t^n_k \wedge t, t^n_{k+1} \wedge t} + [X]^{\pi^n}_t
    \end{split}
  \end{equation}
  for $t \in [0,T]$, where we recall $[X]^{\pi^n}:= \sum_{k=1}^{N_n-1} X_{t^n_k \wedge \cdot, t^n_{k+1} \wedge \cdot} \otimes X_{t^n_k \wedge \cdot, t^n_{k+1} \wedge \cdot}$. Secondly, note that, for any control function $c$, we have
  \begin{equation}\label{eq: bound 1 RIE equivalent to Levy area and quadratic variation}
    \begin{split}
    &\sup_{0 \le k < \ell \le N_n} \frac{|\cL^{\pi^n}_{t^n_k,t^n_\ell}(X) - (X_{t^n_k} + X_{t^n_\ell}) \otimes X_{t^n_k,t^n_\ell}|^{\frac{p}{2}}}{c(t^n_k,t^n_\ell)}\\
    &\quad \lesssim \sup_{0 \le k < \ell \le N_n} \frac{|(\int_0^{\cdot} X_s \otimes \d^{\pi^n} X_s)_{t^n_k,t^n_\ell} - X_{t^n_k} \otimes X_{t^n_k,t^n_\ell}|^{\frac{p}{2}}}{c(t^n_k,t^n_\ell)}\\
    &\qquad + \sup_{0 \le k < \ell \le N_n}  \frac{|X_{t^n_k,t^n_\ell} \otimes X_{t^n_k,t^n_\ell}|^{\frac{p}{2}}}{c(t^n_k,t^n_\ell)} + \sup_{0 \le k < \ell \le N_n}  \frac{| [X]^{\pi^n}_{t^n_k,t^n_\ell} |^{\frac{p}{2}}}{c(t^n_k,t^n_\ell)}
    \end{split}
  \end{equation}
  and
  \begin{equation}\label{eq: bound 2 RIE equivalent to Levy area and quadratic variation}
    \begin{split}
    &\sup_{0 \le k < \ell \le N_n} \frac{|(\int_0^{\cdot} X_s \otimes \d^{\pi^n} X_s)_{t^n_k,t^n_\ell} - X_{t^n_k} \otimes X_{t^n_k,t^n_\ell}|^{\frac{p}{2}}}{c(t^n_k,t^n_\ell)}\\
    &\quad \lesssim \sup_{0 \le k < \ell \le N_n} \frac{|\cL^{\pi^n}_{t^n_k,t^n_\ell}(X) - (X_{t^n_k} + X_{t^n_\ell}) \otimes X_{t^n_k,t^n_\ell}|^{\frac{p}{2}}}{c(t^n_k,t^n_\ell)} \\
    &\qquad + \sup_{0 \le k < \ell \le N_n}  \frac{|X_{t^n_k,t^n_\ell} \otimes X_{t^n_k,t^n_\ell}|^{\frac{p}{2}}}{c(t^n_k,t^n_\ell)}
    + \sup_{0 \le k < \ell \le N_n}  \frac{| [X]^{\pi^n}_{t^n_k,t^n_\ell} |^{\frac{p}{2}}}{c(t^n_k,t^n_\ell)}.
    \end{split}
  \end{equation}
  If $X$ satisfies Property \textup{(RIE)} (i), $(\int_0^\cdot X_s \otimes \d^{\pi^n} X_s)_{n \in \N}$ converges uniformly to $\int_0^\cdot X_s \otimes \d^\pi X_s$ and, by~\cite[Lemma~4.17]{Perkowski2016}, $X$ possesses quadratic variation, that is, $([X]^{\pi^n})_{n \in \N}$ converges uniformly to $[X]^\pi$ as $n \to \infty$. And, again due to~\cite[Lemma~4.17]{Perkowski2016}, $([X]^{\pi^n})_{n \in \N}$ has uniformly bounded $1$-variation. Hence, by~\eqref{eq: Levy area in Riemann integral plus quadratic variation} and~\eqref{eq: bound 1 RIE equivalent to Levy area and quadratic variation}, if $X$ satisfies Property \textup{(RIE)} (i), then $X$ possesses quadratic variation and L{\'e}vy area, both along~$\pi$, and if $X$ satisfies Property \textup{(RIE)} (ii), then $X$ satisfies Assumption~\ref{ass: regularity}.

  Conversely, if $X$ possesses quadratic variation, $([X]^{\pi^n})_{n \in \N}$ converges uniformly to $[X]^\pi$ as $n \to \infty$, and as in the proof of~\cite[Lemma~4.17]{Perkowski2016}, one can show that $([X]^{\pi^n})_{n \in \N}$ has uniformly bounded $1$-variation. If $X$ possesses L{\'e}vy area, $(\cL^{\pi^n}(X))_{n \in \N}$ converges uniformly to $\cL^\pi(X)$ as $n \to \infty$. Hence, by~\eqref{eq: Levy area in Riemann integral plus quadratic variation} and~\eqref{eq: bound 2 RIE equivalent to Levy area and quadratic variation}, if $X$ possesses quadratic variation and L{\'e}vy area, then $X$ satisfies Property \textup{(RIE)} (i), and if $X$ satisfies Assumption~\ref{ass: regularity}, then $X$ satisfies Property \textup{(RIE)} (ii).
\end{proof}

\begin{corollary}\label{cor: gamma RIE eq to Levy and quad var}
  Let $\gamma \in [0,1]$, $\gamma \neq \frac{1}{2}$, $p \in (2,3)$ and let $X$ and $\pi = (\pi^n)_{n\in \N}$ be as in Assumption~\ref{ass: partitions}. Then, $X$ satisfies Property $\gamma$-\textup{(RIE)} (i) relative to $\gamma$ and $\pi$ if and only if $X$ possesses quadratic variation and L{\'e}vy area, both along $\pi$, and $X$ satisfies Property $\gamma$-\textup{(RIE)} (ii) relative to $\gamma$, $p$ and $\pi$ if and only if $X$ satisfies Assumption~\ref{ass: regularity} relative to $p$ and $\pi$.
\end{corollary}

\begin{proof}
  The corollary follows by combining Lemma~\ref{lem: RIE eq to gamma RIE} and~\ref{lem: RIE eq to Levy and quad var}.
\end{proof}

\begin{lemma}\label{lem: Levy eq to 1/2 RIE}
  Let $X$ and $\pi = (\pi^n)_{n\in \N}$ be as in Assumption~\ref{ass: partitions}. Then, $X$ satisfies Property $\gamma$-\textup{(RIE)} (i) relative to $\gamma = \frac{1}{2}$ and $\pi$ if and only if $X$ possesses L{\'e}vy area along $\pi$, and $X$ satisfies Property $\gamma$-\textup{(RIE)} (ii) relative to $\gamma = \frac{1}{2}$, $p$ and $\pi$ if and only if $X$ satisfies Assumption~\ref{ass: regularity} relative to $p$ and~$\pi$.
\end{lemma}

\begin{proof}
  We note that
  \begin{equation}\label{eq: Levy area in gamma integral}
    \begin{split}
    &\cL^{\pi^n}_t(X) := \sum_{k=0}^{N_n-1} (X_{t^n_k \wedge t} + X_{t^n_{k+1} \wedge t}) X_{t^n_k \wedge t, t^n_{k+1} \wedge t} \\
    &\quad = 2 \int_0^t X_s \otimes \d^{\frac{1}{2},\pi^n} X_s + 2 \sum_{k=0}^{N_n-1} (X_{t^n_k, t^n_k \wedge t} + \frac{1}{2} (X_{t^n_{k+1}, t^n_{k+1} \wedge t} - X_{t^n_k,t^n_k \wedge t})) X_{t^n_k \wedge t, t^n_{k+1} \wedge t}
    \end{split}
  \end{equation}
  for $t \in [0,T]$. And therefore, for any control function $c$, we have that
  \begin{equation}\label{eq: bound 1 Levy area}
    \begin{split}
    &\sup_{0 \le k < \ell \le N_n} \frac{|(\int_0^\cdot X_s \otimes \d^{\frac{1}{2},\pi^n} X_s)_{t^n_k,t^n_\ell} - X_{t^n_k} \otimes X_{t^n_k,t^n_\ell}|^{\frac{p}{2}}}{c(t^n_k,t^n_\ell)} \\
    &\quad \leq \sup_{0 \le k < \ell \le N_n} \frac{|\cL^{\pi^n}_{t^n_k,t^n_\ell}(X) - (X_{t^n_k} + X_{t^n_\ell}) \otimes X_{t^n_k,t^n_\ell}|^{\frac{p}{2}}}{c(t^n_k,t^n_\ell)} + \sup_{0 \le k < \ell \le N_n} \frac{|X_{t^n_k,t^n_\ell} \otimes X_{t^n_k,t^n_\ell}|^{\frac{p}{2}}}{c(t^n_k,t^n_\ell)}
    \end{split}
  \end{equation}
  and
  \begin{equation}\label{eq: bound 2 Levy area}
    \begin{split}
    &\sup_{0 \le k < \ell \le N_n} \frac{|\cL^{\pi^n}_{t^n_k,t^n_\ell}(X) - X_{t^n_k,t^n_\ell} \otimes X_{t^n_k,t^n_\ell}|^{\frac{p}{2}}}{c(t^n_k,t^n_\ell)} \\
    &\quad \lesssim \sup_{0 \le k < \ell \le N_n} \frac{|(\int_0^\cdot X_s \otimes \d^{\frac{1}{2},\pi^n} X_s)_{t^n_k,t^n_\ell} - X_{t^n_k} \otimes X_{t^n_k,t^n_\ell}|^{\frac{p}{2}}}{c(t^n_k,t^n_\ell)} + \sup_{0 \le k < \ell \le N_n} \frac{|X_{t^n_k,t^n_\ell} \otimes X_{t^n_k,t^n_\ell}|^{\frac{p}{2}}}{c(t^n_k,t^n_\ell)}.
    \end{split}
  \end{equation}
  If $X$ possesses L{\'e}vy area,  $(\cL^{\pi^n}(X))_{n \in \N}$ converges uniformly to $\cL^\pi(X)$, and if $X$ satisfies Property $\gamma$-\textup{(RIE)} (i) relative to $\gamma = \frac{1}{2}$, $(\int_0^{\cdot} X_s \otimes \d^{\frac{1}{2}, \pi^n} X_s)_{n\in \N}$ converges uniformly to $\int_0^{\cdot} X_s \otimes \d^{\frac{1}{2},\pi} X_s$. Hence, by~\eqref{eq: Levy area in gamma integral} and~\eqref{eq: bound 1 Levy area}, if $X$ possesses L{\'e}vy area, then $X$ satisfies Property $\gamma$-\textup{(RIE)} (i), and if $X$ satisfies Assumption~\ref{ass: regularity}, then $X$ satisfies Property $\gamma$-\textup{(RIE)} (ii) relative to $\gamma = \frac{1}{2}$. Conversely, by~\eqref{eq: Levy area in gamma integral} and~\eqref{eq: bound 2 Levy area}, if $X$ satisfies Property $\gamma$-\textup{(RIE)} (i) relative to $\gamma = \frac{1}{2}$, then $X$ possesses L{\'e}vy area, and if $X$ satisfies Property $\gamma$-\textup{(RIE)} (ii) relative to $\gamma = \frac{1}{2}$, then $X$ satisfies Assumption~\ref{ass: regularity}.
\end{proof}

\section{Invariance of quadratic variation, L{\'e}vy area, and pathwise integration}\label{sec: invariance}

It is well observed that the pathwise quadratic variation, pathwise L{\'e}vy area and, thus, the resulting pathwise integrals may depend on the chosen sequence of partitions, see e.g. \cite{Freedman1983,Davis2018,Friz2020}. For previous works on the invariance of quadratic variation we refer to \cite{contdas2023, cont-das2022}. In this section we derive conditions for the invariance of quadratic variation and L{\'e}vy area with respect to the choice of the sequence of partitions. To that end, we need to introduce some definitions.

\begin{definition}[Super-sequence and sub-sequence]
  Let $\pi = (\pi^n)_{n \in \N}$ be a sequence of partitions of the interval $[0,T]$. A sequence $(\sigma^n)_{n \in \N}$ of partitions is called a \emph{super-sequence} of $\pi = (\pi^n)_{n \in \N}$ if there exists a non-decreasing map $r \colon \N \to \N$ such that $\sigma^n = \pi^{r(n)}$, $n \geq r(n)$ for all $n \in \N$, and $r(n) \to \infty$ as $n \to \infty$. Analogously, a sequence $(\sigma^n)_{n \in \N}$ of partitions is called a \emph{sub-sequence} of $\pi = (\pi^n)_{n \in \N}$ if there exists a non-decreasing map $r \colon \N \to \N$ such that $\sigma^n = \pi^{r(n)}$ and $n \leq r(n)$ for all $n \in \N$.
\end{definition}

Note that the quadratic variation or L{\'e}vy area of any path $X\in C([0,T];\R^d)$ along a sequence $\pi$ of partitions is always the same along any sub-sequence or super-sequence $\sigma$ of $\pi$. To derive sufficient and necessary conditions on the invariance of quadratic variation or L{\'e}vy area, we rely again on the concept of a control function. We recall that a continuous path $X \in C([0,T];\R^d)$ has finite $p$-variation, i.e.~$\|X\|_p < \infty$, if and only if there exists a control function $c$ such that
\begin{equation*}
  \sup_{(u,v) \in \Delta_T} \frac{|X_v - X_u|^p}{c(u,v)} < \infty.
\end{equation*}

In this section, we write $N(\pi^n)$ instead of $N_n$ for the number of partition points in the partition $\pi^n$ to emphasize the dependence of the partition $\pi = (\pi^n)_{n \in \N}$.

\begin{definition}[$c$-balanced sequence of partitions]
  Let $\pi= (\pi^n)_{n \in N}$, with $\pi^n = \{0 = t^n_0 < t^n_1 < \dots < t^n_{N(\pi^n)} = T\}$, $n \in \N$, be a sequence of partitions of $[0,T]$, and $c \in C(\Delta_T;\R)$ be a control function. We say that $\pi$ is a {\em $c$-balanced sequence of partitions} if
  \begin{equation*}
    \exists C>0 \qquad \text{such that} \qquad \forall n \in \N, \quad \frac{\sup_{t^n_i \in \pi^n}  c(t^n_i, t^n_{i+1})}{\inf_{t^n_i \in \pi^n}  c(t^n_i, t^n_{i+1})} \leq C.
  \end{equation*}
\end{definition}

\begin{remark}
  Choosing the control $c(s,t) := |t-s|$, a $c$-balanced sequence $\pi$ of partitions is a balanced sequence of partitions in the sense of \cite[Definition~2.1]{contdas2023}. Indeed, we recall that a $\pi = (\pi^n)_{n \in \N}$ is called {\em balanced} if
  \begin{equation*}
    \exists C>0 \qquad \text{such that} \qquad \forall n \in \N, \quad \frac {|\pi^n|}{\underline{\pi^n}}\leq C,
  \end{equation*}
  where $\pi^n = \{0 = t^n_0 < t^n_1 < \dots < t^n_{N(\pi^n)} = T\}$, $n \in \N$, and
  \begin{equation*}
    \underline{\pi^n} := \inf_{i=0,\dots,N(\pi^n)-1} |t^n_{i+1}-t^n_i|,\qquad |\pi^n| := \sup_{i=0,\dots,N(\pi^n)-1} |t^n_{i+1}-t^n_i|.
  \end{equation*}
\end{remark}

\subsection{Invariance of quadratic variation}

We begin by generalizing the notion of \emph{quadratic roughness} \cite[Definition~3.1]{contdas2023} and the results concerning the invariance of quadratic variation \cite[Lemma~4.1 and Theorem~4.2]{contdas2023}, allowing for general control functions. Consequently, these extensions allow to treat continuous paths of finite $p$-variation, and not only the H{\"o}lder continuous paths. This then particularly includes sample paths of semimartingales almost surely, in contrast to~\cite{contdas2023}. 

To that end, let $\T = (\T^n)_{n \in \N}$ be the sequence of dyadic partitions, that is,
\begin{equation*}
  \T^n := \{ 0 < T2^{-n} < 2T2^{-n} < 3T2^{-n} < \cdots < T\}, \quad n \in \N.
\end{equation*}

\begin{definition}[Quadratic roughness (with respect to a control)]
  Let $p \in (2,3)$, $c$ be a control function, and $X \in C([0,T];\R^d)$ be a continuous path such that
  \begin{equation*}
    \sup_{(s,t) \in \Delta_T}\frac{|X_{s,t}|^p}{c(s,t)} \lesssim  1.
  \end{equation*}
  Moreover, suppose that $\T = (\T^n)_{n \in \N}$ is a $c$-balanced sequence of partitions and $\pi = (\pi^n)_{n \in \N}$, with $\pi^n = \{0 = t^n_0 < t^n_1 < \cdots < t^n_{N(d^n)} = T\}$, $n \in \N$, is a sequence of partitions with vanishing mesh size. We say that $X\in Q^\T([0,T];\R^d)$ has \emph{quadratic roughness} with index~$p$ along $\pi$ on $[0,T]$ if the following holds. There exists a sub-sequence or super-sequence $d^n= \{0=s^n_0<s^n_1<\cdots<s^n_{N(d^n)}=T\}$ of $\T$ with the following properties:
  \begin{enumerate}
    \item [(i)] $\frac{N(\pi^n)^p}{N(d^n)}\longrightarrow 0$ as $n \longrightarrow \infty$,
    \item [(ii)] for all $t\in [0,T]$,
      \begin{equation*}
        \sum_{k=0}^{N(\pi^n)-1} \sum_{s^n_i\neq s^n_{i'} \in (t^n_k,t^n_{k+1}]} X_{s^n_i \wedge t, s^n_{i+1} \wedge t} \otimes X_{s^n_{i'} \wedge t, s^n_{i'+1} \wedge t} \enspace \longrightarrow 0 \quad \text{as} \quad n \longrightarrow \infty.
      \end{equation*}
  \end{enumerate}
\end{definition}

Note that this definition of quadratic roughness does not require the partition $\pi$ to be balanced as in \cite{contdas2023}, but instead requires the `reference partition' $\T$ to be $c$-balanced with respect to a control function $c$.

\begin{proposition}[Invariance of quadratic variation]\label{prop: invariance of quad variation}
  Let $p \in (2,3)$, $c$ be a control function, and $X\in C([0,T];\R^d)$ be a continuous path such that
  \begin{equation*}
    \sup_{(s,t) \in \Delta_T} \frac{|X_{s,t}|^p}{c(s,t)} \lesssim  1.
  \end{equation*}
  Moreover, suppose that $\T = (\T^n)_{n \in \N}$ is a $c$-balanced sequence of partitions and $\pi =(\pi^n)_{n\in \N}$ is a sequence of partitions with vanishing mesh size. Then, $X$ has quadratic roughness with index $p$ along $\pi$ on $[0,T]$ if and only if
  \begin{equation*}
    X \in Q^\pi([0,T];\R^d) \quad \text{and}, \quad \forall  t \in [0,T], \quad [X]^\pi_t = [X]^\T_t.
  \end{equation*}
\end{proposition}

We will omit the proof of Proposition~\ref{prop: invariance of quad variation}, as the proof follows a very similar line of argument of that in \cite[Lemma~4.1 and Theorem~4.2]{contdas2023}.

\subsection{Invariance of L{\'e}vy area}
If a process $X$ has L{\'e}vy area along a partition sequence $\pi$ with vanishing mesh size, then the process $X$ also has L{\'e}vy area along any subsequence $(\pi^{l_n})_{n \in \N}$ of $(\pi^n)_{n \in \N}$. This simply means that $\mathcal \sup_{t\in [0,T]}|L^{\pi^{n}}_t(X)- L^{\pi^{l_n}}_t(X)| \to 0$ as $n \to \infty$. If one chooses the subsequence $l_n$ such that $l_n/n$ diverges to $\infty$ fast enough, the above condition can be rewritten as follows:
\begin{equation}\label{eq.levycond.motivation}
  \begin{split}
  &\sum_{k=0}^{N(\pi^n)-1}\sum_{t^{l_n}_i, t^{l_n}_{i'}} (-1)^{i-\inf_l\{t^{l_n}_l>t^n_k\}} (X_{t^{l_n}_{i+1} \wedge t}+X_{t^{l_n}_i \wedge t})\otimes X_{t^{l_n}_{i'} \wedge t, t^{l_n}_{i'+1} \wedge t}\\
  &\quad \quad \quad- \sum_{k=0}^{N(\pi^n)-1}\sum_{t^{l_n}_i} (X_{t^{l_n}_{i+1}\wedge t}+X_{t^{l_n}_{i}\wedge t})\otimes X_{t^{l_n}_{i}\wedge t,t^{l_n}_{i+1}\wedge t} \enspace \longrightarrow 0 \quad \text{as} \quad n \longrightarrow \infty,
  \end{split}
\end{equation}
where the inner two sums are taken such that $ t^{l_n}_i,t^{l_n}_{i'}\in (t^n_k,t^n_{k+1}]$, the two consecutive partition points of $\pi^n$ (with some slight abuse of notation, for details see  \eqref{eq.levycond.modified} and the details thereafter). As $n\to \infty$, the number of terms in the two inner sum in \eqref{eq.levycond.motivation} blows up. So this convergence to $0$ is a consequence of enough cancellation across neighboring increments such that their cross-area of the form $(X_{t^{l_n}_i}+X_{t^{l_n}_{i+1}})\otimes X_{t^{l_n}_{i'},t^{l_n}_{i'+1}}$  average to zero within two consecutive partition points of the slower partition $(\pi^n)_{n \in \N}$. This happens because the increments of $X$ have enough alternating signs over any interval, which is an indicator of the roughness of the process $X$. We now introduce a slightly extended version of \eqref{eq.levycond.motivation} called \emph{L{\'e}vy roughness} to characterize the invariance of L{\'e}vy area.

\begin{definition}[L{\'e}vy roughness]\label{def: levy roughness}
  Let $p \in (2,3)$, $c$ be a control function and $X\in C([0,T];\R^d)$ be a continuous path such that
  \begin{equation*}
    \sup_{(s,t) \in \Delta_T} \frac{|X_{s,t}|^p}{c(s,t)} \lesssim  1.
  \end{equation*}
  Moreover, suppose that $\T = (\T^n)_{n \in \N}$ is a $c$-balanced sequence of partitions and $\pi=(\pi^n)_{n \in \N}$, with $\pi^n = \{0 = t^n_0 < t^n_1 < \cdots < t^n_{N(d^n)} = T\}$, $n \in \N$, is a sequence of partitions with vanishing mesh size. We say that $X \in \cL^\T([0,T];\R^d)$ has {\em L{\'e}vy roughness} with index~$p$ along $\pi$ on $[0,T]$ if the following holds. There exists a sub-sequence or super-sequence $\allowdisplaybreaks d^n = \{0 = s^n_0 < s^n_1 < \cdots < s^n_{N(d^n)} = T\}$ of $\T$ with the following properties:
  \begin{enumerate}
    \item [(i)] $\frac{N(\pi^n)^p}{N(d^n)} \longrightarrow 0$ as $n \longrightarrow \infty$,
    \item [(ii)] for all $t\in [0,T]$:
          \begin{equation}\label{eq.levycond.modified}
            \begin{split}
			&\sum_{k=0}^{N(\pi^n)-1}\sum_{s^n_i, s^n_{i'}} (-1)^{i-\inf_l\{s^n_l>t^n_k\}} (X_{s^n_{i+1} \wedge t}+X_{s^n_i \wedge t})\otimes X_{s^n_{i'} \wedge t, s^n_{i'+1} \wedge t}\\
            &\quad \quad \quad- \sum_{k=0}^{N(\pi^n)-1}\sum_{s^n_i} (X_{s^n_{i+1}\wedge t}+X_{s^n_{i}\wedge t})\otimes X_{s^n_{i}\wedge t,s^n_{i+1}\wedge t} \enspace \longrightarrow 0 \quad \text{as} \quad n \longrightarrow \infty,
            \end{split}
          \end{equation}
          where the inner two sums are taken as $i,i' \text{ such that } s^n_i,s^n_{i'}\in (t^n_k,t^n_{k+1}] \text{ if } \inf_l\{s^n_l>t^n_{k+1}\}-\inf_l\{s^n_l>t^n_k\} \text{ is an odd number}$ and $s^n_i,s^n_{i+1},s^n_{i'},s^n_{i'+1}\in (t^n_k,t^n_{k+1}]  \text{ if } \inf_l\{s^n_l>t^n_{k+1}\}-\inf_l\{s^n_l>t^n_k\} \text{ is an even number}$.
  \end{enumerate}
\end{definition}

\begin{remark}
  The inner summations of \eqref{eq.levycond.modified} are taken in such a way that one always has an odd number of terms in the inner sum. In case there is an even number of $s^n_i\in (t^n_k,t^n_{k+1}]$, we ignore the last (largest) $s^n_.\in (t^n_k,t^n_{k+1}]$ in the inner summation of \eqref{eq.levycond.modified}.
\end{remark}

\begin{remark}
  The L{\'e}vy roughness condition \eqref{eq.levycond.modified} only depends on the functional value of $X$ at the dyadic partition points $s^n_i$ and does not depend on the functional value of $X$ at the partition points of $\pi$. So in practice, to verify L{\'e}vy roughness, one only needs to observe the underlying function/process $X$ along the reference dyadic partition.
\end{remark}

We note that \eqref{eq.levycond.modified} in Definition \ref{def: levy roughness} can be rewritten as follows: 
\begin{equation}\label{eq.levycond}
  \begin{split}
  &\sum_{k=0}^{N(\pi^n)-1}\sum_{s^n_i\neq s^n_{i'}\in (t^n_k,t^n_{k+1}]} (-1)^{i-\inf_l\{s^n_l>t^n_k\}} (X_{s^n_{i+1}\wedge t}+X_{s^n_{i} \wedge t})\otimes X_{s^n_{i'} \wedge t, s^n_{i'+1}\wedge t} \\
  &- \sum_{k=0}^{N(\pi^n)-1}\sum_{s^n_i \in (t^n_k,t^n_{k+1}]} (1-(-1)^{i-\inf_l\{s^n_l>t^n_k\}}) (X_{s^n_{i+1} \wedge t} + X_{s^n_i \wedge t}) \otimes X_{s^n_i \wedge t, s^n_{i+1} \wedge t}\\
  &+ \sum_{k=0}^{N(\pi^n)-1} \1_{\{\inf_l\{s^n_l>t^n_{k+1}\}-\inf_l\{s^n_l>t^n_{k}\}=\text{odd}\}}(X_{s^n_{\inf_l\{s^n_l>t^n_{k+1}\}-1} \wedge t} + X_{s^n_{\inf_l\{s^n_l>t^n_{k+1}\}} \wedge t})  \\
  &\hspace{20em} \otimes \sum_{s^n_i\in (t^n_k,t^n_{k+1}]} X_{s^n_i \wedge t,s^n_{i+1} \wedge t} \\
  &\longrightarrow 0 \quad \text{as} \quad n \longrightarrow \infty.
  \end{split}
\end{equation}
The third summation term in \eqref{eq.levycond} can be considered as a correction term. It only appears when within two consecutive partition points of $\pi^n$, an even number of partition points of $d^n$ are present. Additionally, the second summation term in \eqref{eq.levycond} vanishes exactly half of the time.

\begin{remark}
  Condition (i) of Definition~\ref{def: levy roughness} implies that $|d^n|^{\frac{1}{p}} = o(|\pi^n|)$. Furthermore, if we assume that $\pi$ is a balanced sequence of partitions, then
  \begin{equation*}
    N(\pi^n) = o(d^n) \iff |d^n|^\frac{1}{p} = o(|\pi^n|).
  \end{equation*}
\end{remark}

The following theorem shows that the L{\'e}vy roughness is a necessary and sufficient condition for the L{\'e}vy area to be invariant with respect to the choice of partition sequences.

\begin{theorem}[Invariance of L{\'e}vy area]\label{thm:levy roughness}
  Let $p \in(2,3)$, $c$ be a control function, and $X \in C([0,T];\R^d)$ be a continuous path such that
  \begin{equation*}
    \sup_{(s,t) \in \Delta_T}\frac{|X_{s,t}|^p}{c(s,t)} \lesssim  1.
  \end{equation*}
  Moreover, suppose that $\T = (\T^n)_{n \in \N}$ is a $c$-balanced sequence of partitions and $\pi=(\pi^n)_{n\in \N}$ is a sequence of partitions with vanishing mesh size. Then, $X$ has L{\'e}vy roughness with index~$p$ along $\pi$ if and only if
  \begin{equation*}
    X \in \cL^\pi([0,T];\R^d) \quad \text{and}, \quad  \forall  t\in [0,T], \quad \cL^\pi_t(X) = \cL^\T_t(X).
  \end{equation*}
\end{theorem}

We will only prove this result for the terminal time~$T$. For a general $t\in[0,T]$, the proof follows a similar line of argument.

\begin{proof}
  \emph{Part~1: Sufficiency of L{\'e}vy roughness.} Assume $D = (d^n)_{n \in \N}$ be the sub-/super-sequence of the sequence of dyadic partitions $\T$ along which $X$ has L{\'e}vy roughness and, in particular, \eqref{eq.levycond.modified} holds. Since $D= (d^n)_{n \in \N}$ is a sub-/super-sequence of the dyadic $\T$, $D$ is also a $c$-balanced sequence of partitions of $[0,T]$. The points of $\pi^n$ are interspersed along those of $d^n$. The generic partition points of $\pi^n$ are denoted as $t^n_k$, and the generic partition points of $d^n$ are denoted as $s^n_i$.

  For $k = 1, 2, \dots, N(\pi^n) - 1$, we define
  \begin{equation*}
    p(n,k) := \inf\{i \geq 1: \, s^n_i \in (t^n_k,t^n_{k+1}] \}.
  \end{equation*}
  Then, we get the following inequality regarding the partition points of $\pi^n$ and $d^n$:
  \begin{equation*}
    s^n_{p(n,k)-1} \leq t^n_k < s^n_{p(n,k)}  <  \cdots < s^n_{p(n,k+1)-1} \leq t^n_{k+1} <  s^n_{p(n,k+1)},
  \end{equation*}
  where we set $p(n,N(\pi^n)) - 1 = N(d^n)$ and $p(n,0) = 1 $.

  Let us first define an auxiliary partition
  \begin{equation*}
    \sigma^n = \{0 = s^n_0 < s^n_{p(n,1)} < s^n_{p(n,2)} < \cdots < s^n_{p(n,N(\pi^n))-1} =T\}.
  \end{equation*}
  So this new sequence of partition $\sigma$ is ``very close'' and ``slightly on the right side'' of the partition sequence $\pi$ and every partition point of $\sigma^n$ is a partition point of $d^n$ and, hence, a partition point of the dyadic partition of the form $\frac{i}{2^l}$ for some suitable $i, l \in \N$.
  
  \smallskip
  Now let us consider two cases:

  \emph{Case 1}: $p(n,k+1) - p(n,k) = \text{odd}$. In this case we decompose $X_{t^n_k} + X_{t^n_{k+1}}$ as follows:
  \begin{equation*}
    \underbrace{X_{t^n_k}+X_{t^n_{k+1}}}_{=: \Delta^{+}_{n,k}}= \underbrace{X_{s^n_{p(n,k+1)},t^n_{k+1}}}_{=: D_k} +
    \underbrace{ X_{s^{n}_{p(n,k)},t^n_k}}_{=: B_k}
    + \underbrace{ \sum_{i=p(n,k)}^{p(n,k+1)-1} (-1)^{i-p(n,k)}(X_{s^n_{i+1}}+X_{s^n_i}) }_{=: C^+_k}.
  \end{equation*}

  \smallskip
  \emph{Case 2}: $p(n,k+1) - p(n,k) = \text{even}$. In this case we write:
  \begin{equation*}
    \underbrace{X_{t^n_{k}}+X_{t^n_{k+1}}}_{=: \Delta^{+}_{n,k}}= \underbrace{X_{s^n_{p(n,k+1)-1} ,t^n_{k+1}}}_{=: D_k} +
    \underbrace{X_{s^n_{p(n,k)}, t^n_k}}_{=: B_k}
    + \underbrace{ \sum_{i=p(n,k)}^{p(n,k+1)-2} (-1)^{i-p(n,k)}(X_{s^n_{i+1}}+X_{s^n_i} ) }_{=: C^+_k}.
  \end{equation*}
  Note that, in Case~1 and Case~2, the quantity $C_k^+$ and $D_k$ are slightly different. The difference between Case~1 and Case~2 for both $C_k^+$ and $D_k$ is exactly $\delta_k:= X_{s^n_{p(n,k+1)-1}, s^n_{p(n,k+1)}}$. We can think of the term $\delta_k$ as a correction term appearing in Case 2 (even case), which is a ``very small'' increment term (increment of length smaller than or equal to $|d^n|$), and we will show that this term can be controlled as a consequence of the existence of the control function~$c$.

  Without loss of generality, we will assume $X_0 = 0$. We can observe that for both cases,
  \begin{align*}
    &|D_k| \lesssim \max_{(s,t) \in d^n} c(s,t)^{\frac{1}{p}}, \quad |B_k| \lesssim \max_{(s,t) \in d^n} c(s,t)^{\frac{1}{p}}, \\
    &|C^+_k| \lesssim c(0,T)^{\frac{1}{p}} < \infty, \quad |\Delta^+_{n,k}| \lesssim c(0,T)^{\frac{1}{p}} < \infty.
  \end{align*} 
  Similarly to the proof of~\cite[Theorem~4.2]{contdas2023}, in Case $1$ above we write,
  \begin{equation*}
    \underbrace{X_{t^n_k,t^n_{k+1}}}_{=: \Delta^-_{n,k}}=- \underbrace{X_{s^n_{p(n,k)},t^n_{k}}}_{= B_k}+ \underbrace{X_{s^n_{p(n,k+1)},t^n_{k+1}}}_{= D_k}  
    + \underbrace{ \sum_{i=p(n,k)}^{p(n,k+1)-1} X_{s^n_i,s^{n}_{i+1}}}_{=: C^-_k}.
  \end{equation*}
  In Case 2 we slightly modify the decomposition of $X_{t^n_{k},t^n_{k+1}}$ from the proof of~\cite[Theorem~4.2]{contdas2023} to match $D_k$ in Case 2 (which is slightly different from $D_k$ in Case 1) and write
  \begin{equation*}
    \underbrace{X_{t^n_k,t^n_{k+1}}}_{= \Delta^-_{n,k}}= - \underbrace{X_{s^n_{p(n,k)},t^n_k}}_{= B_k} +\underbrace{X_{s^n_{p(n,k+1)-1},t^n_{k+1}}}_{= D_k}
    + \underbrace{ \sum_{i=p(n,k)}^{p(n,k+1)-2} X_{s^n_i,s^n_{i+1}}}_{=C^-_k}.
  \end{equation*}
  We can observe that for both cases, 
  \begin{equation*}
  |C^-_k| \lesssim \max_{t^n_k \in \pi^n} c(t^n_k,s^n_{p(n,i+1)})^{\frac{1}{p}}\lesssim \max_{t^n_k \in \pi^n} c(t^n_k,t^n_{k+2})^{\frac{1}{p}}.
  \end{equation*}
  Now the difference of L{\'e}vy areas along the two partition sequences $\pi^n$ and $\d^n$ at level $n$ can be written as follows:
  \begin{equation}\label{eq.inv}
    \begin{split}
    &\cL^{\pi^n}_T(X) - \cL^{d^n}_T(X)\\
    &\quad = \sum_{k=0}^{N(\pi^n)-1} \bigg[{\Delta^+_{n,k}}\otimes {\Delta^-_{n,k}} - \sum_{i=p(n,k)}^{p(n,k+1)-1}  (X_{s^n_{i+1}}+X_{s^n_i})\otimes X_{s^n_i,s^n_{i+1}} \bigg] \\
    &\quad=\sum_{k=0}^{N(\pi^n)-1} \Big[{\Delta^+_{n,k}} \otimes {\Delta^-_{n,k}}- C_k^+ \otimes C_k^- \Big] \\
    &\quad\quad+ \sum_{k=0}^{N(\pi^n)-1} \bigg[C_k^+ \otimes C_k^- -\sum_{i=p(n,k)}^{p(n,k+1)-1}  (X_{s^n_{i+1}}+ X_{s^{n}_{i}})\otimes X_{s^n_i,s^n_{i+1}} \bigg].
    \end{split}
  \end{equation}
  We will first show that the norm of the first term, namely,
  \begin{equation*}
    \bigg|\sum_{k=1}^{N(\pi^n)-1} \Big[\Delta^+_{n,k} \otimes \Delta^-_{n,k} - C_k^+\otimes C_k^- \Big] \bigg|
  \end{equation*}
  converges to $0$ as $n \to \infty$ as a consequence of $\T$ being $c$-balanced control. Then, we will show that the norm of the second term in equation~\eqref{eq.inv} will converge to $0$ as $n \to \infty$ when $X$ has L{\'e}vy roughness, which will then complete the proof of this theorem. \\
  \emph{Step~1:} So, it holds that,
  \begin{equation}\label{Eq.firstterm}
    \begin{split}
    &|\Delta^+_{n,k} \otimes \Delta^-_{n,k} - C_k^+ \otimes C_k^-|\\
    &\quad = |(D_k+B_k+C_k^+)\otimes (\pm B_k+D_k+C_k^-) -  C_k^+ \otimes C_k^-|\\
    &\quad=|(D_k+B_k)\otimes (\pm B_k+D_k) + C^+_k \otimes (\pm B_k+D_k) +(D_k+B_k) \otimes C_k^-|
    \end{split}.
  \end{equation}
  The $\pm$ sign on the second sum is chosen to be $-$ve and $+$ve according to the two odd and even cases above. $D_k$, $B_k$ are increments of $X$, where the increments are taken at most of the length $|d^n|$ apart, and we took the sum over $\pi^n$. Since $\T$ is $c$-balanced, we get that
  \begin{align*}
  \bigg|\sum_{k=0}^{N(\pi^n)-1} (D_k+B_k) \otimes(\pm B_k + D_k) \bigg|
    &\leq \sum_{k=0}^{N(\pi^n)-1}  4 \max_{[s,t]\in d^n} |c(s,t)|^{\frac{2}{p}}  \\
    &\lesssim N(\pi^n) \min_{[s,t] \in d^n} |c(s,t)|^{\frac{2}{p}}
    \lesssim \frac{N(\pi^n)}{N(d^n)^{\frac{2}{p}}} \, c(0,T)^{\frac{2}{p}} \xrightarrow{n \to \infty} 0.
  \end{align*}
  In the second inequality, we use that the dyadic sub-/super-sequence $d = (d^n)_{n \in \N}$ is also $c$-balanced and the limit follows from condition (i) of Definition~\ref{def: levy roughness}.

  Now if we can show that $|\sum_{\pi^n} C_k^+ \otimes D_k|\to 0$ as $n \to \infty$, in a similar line of arguments, one will be able to show that all the other terms of $\sum_{k=1}^{N(\pi^n)-1} [\Delta^+_{n,k} \otimes \Delta^-_{n,k} - C_k^+ \otimes  C_k^- ] $ from \eqref{Eq.firstterm} also converge to $0$ in norm as $n \to \infty$ (we will skip the proofs of the other three terms).

  Since $X$ is a continuous function and we only focus on a compact interval $[0,T]$, we can conclude that $X$ is bounded from above. Since we assume $X_0 = 0$, we have, $\sup_{t\in [0,T]} |X_t| = \sup_{t\in [0,T]} |X_{0,t}| \lesssim c(0,T)^{\frac{1}{p}} < \infty$. So,
  \begin{align*}
    \Big |\sum_{k=0}^{N(\pi^n)-1}{C_k^+}\otimes D_k \Big|
    &\leq \sum_{k=0}^{N(\pi^n)-1} 2 c(0,T)^{\frac{1}{p}} \cdot \max_{[s,t] \in d^n} c(s,t)^{\frac{1}{p}} \\
    &= N(\pi^n) \cdot 2 c(0,T)^{\frac{1}{p}} \cdot \max_{[s,t]\in d^n} c(s,t)^{\frac{1}{p}} \\
    & \lesssim  N(\pi^n) \min_{[s,t]\in d^n} c(s,t)^{\frac{1}{p}}
    \lesssim \frac{N(\pi^n)}{N(d^n)^{\frac{1}{p}}} \, c(0,T)^{\frac{1}{p}}\xrightarrow[]{n\to\infty} 0.
  \end{align*}
  The last limit comes from condition (i) of Definition~\ref{def: levy roughness}. This concludes that the norm of the first sum of equation \eqref{eq.inv} converges to $0$ as $n \to \infty$.

  \emph{Step 2: }Let us now focus on the second sum of Equation~\eqref{eq.inv}. We further define (ignoring $n$ which is somewhat fixed unless we take the limit as $n \to \infty$) for notational convenience
  \begin{equation*}
    a_i := X(s^n_{i+1}) + X(s^n_i) \quad \text{and} \quad b_i := X(s^n_{i+1})- X(s^n_i).
  \end{equation*}
  So, it follows that
  \begin{equation}\label{eq. calculation}
  \begin{split}
    &\sum_{k=0}^{N(\pi^n)-1} \bigg[C_k^+ \otimes C_k^- - \sum_{i=p(n,k)}^{p(n,k+1)-1} (X_{s^n_{i+1}} + X_{s^n_i})\otimes X_{s^n_i,s^n_{i+1}}  \bigg] \\
    &= \sum_{k=0}^{N(\pi^n)-1} \bigg[C_k^+\otimes  C_k^- - \sum_{i=p(n,k)}^{p(n,k+1)-1} a_i \otimes b_i \bigg] \\
    &= \sum_{k=0}^{N(\pi^n)-1} \bigg[\bigg(\sum_{i=p(n,k)}^{p(n,k+1)-q_{n,k}} (-1)^{i-p(n,k)} a_i \bigg) \otimes  \bigg(\sum_{i=p(n,k)}^{p(n,k+1)-q_{n,k}} b_i\bigg) -\sum_{i=p(n,k)}^{p(n,k+1)-1} a_i \otimes b_i \bigg] \\
    &=\sum_{k=0}^{N(\pi^n)-1} \bigg[\bigg(\sum_{i=p(n,k)}^{p(n,k+1)-q_{n,k}} (-1)^{i-p(n,k)} a_i \bigg) \otimes \bigg(\sum_{i=p(n,k)}^{p(n,k+1)-q_{n,k}} b_i\bigg) -\sum_{i=p(n,k)}^{p(n,k+1)-q_{n,k}}  a_i \otimes b_i \bigg]  \\
    &\hspace{8cm} + \sum_{k=0}^{N(\pi^n)-1} a_{p(n,k+1)-1}\otimes b_{p(n,k+1)-1} \1_{\{q_{n,k} = 2\}},
    \end{split}
  \end{equation}
  where
  \begin{equation*}
    q_{n,k} = \1_{\{(p(n,k+1)-p(n,k) = \text{ odd}\}} + 2 \1_{\{p(n,k+1)-p(n,k) = \text{ even}\}}.
  \end{equation*}
  It holds for the first term above that
  \begin{equation*}
    \sum_{k=0}^{N(\pi^n)-1} \bigg[\sum_{i, j=p(n,k)}^{p(n,k+1)-q_{n,k}} (-1)^{i-p(n,k)} a_i \otimes  b_j - \sum_{i=p(n,k)}^{p(n,k+1)-q_{n,k}} a_i \otimes b_i \bigg] \longrightarrow 0  \quad \text{as} \quad n \to \infty,
  \end{equation*}
  from the definition of L{\'e}vy roughness in \eqref{eq.levycond.modified}. Furthermore, the remaining term in \eqref{eq. calculation} can be bounded as
  \begin{align*}
    \bigg |\sum_{k=0}^{N(\pi^n)-1} \1_{\{q_{n,k} = 2\}} a_{p(n,k+1)-1} \otimes b_{p(n,k+1)-1}\bigg|
    &\leq \sum_{k=0}^{N(\pi^n)-1} |a_{p(n,k+1)-1} \otimes b_{p(n,k+1)-1}|\\
    &\lesssim N(\pi^n) \cdot c(0,T)^{\frac{1}{p}} \cdot \max_{[s,t]\in d^n} c(s,t)^{\frac{1}{p}},
  \end{align*}
  and converges to $0$ as $n \to \infty$, which follows from $\frac{N(\pi^n)}{N(d^n)^{\frac{1}{p}}} \to 0$ as $n \to \infty$. This shows that the L{\'e}vy area of $X$ along $\pi$ exists and is equal to the L{\'e}vy area of $X$ along the sub-/super-sequence $D$ of $\T$, which is the same as the L{\'e}vy area of $X$ along $\T$, completing the first part of the proof.

  \emph{Part 2: Necessity of L{\'e}vy roughness.} We will first construct a sequence of partitions $D=(d^n)_{n\in\N}$ along which we will show that the L{\'e}vy roughness condition~\eqref{eq.levycond.modified} does hold. Such a sequence of partitions $D=(d^n)_{n\in\N}$ needs to satisfy condition (i) of Definition~\ref{def: levy roughness}. Now fix any $p \in (2,3)$, take any $\varepsilon > 0$ and assume
  \begin{equation*}
    \pi^n = \{0 = t^n_0 < t^n_1 < \cdots < t^n_{N(d^n)} = T\}.
  \end{equation*}
  Since $X$ has L{\'e}vy area along $\T$, it indeed has L{\'e}vy area along any sub-/super-sequence of $\T$. Since both $\T$ and $\pi$ have vanishing mesh size and $N(\T^n) = 2^n \uparrow \infty$ as $n \to \infty$, one can construct a sub-/super-sequence $D=(d^n)_{n\in\N}$ of $\T$ which satisfies condition (i) of Definition \ref{def: levy roughness}. Firstly, if $N(\pi^n) = o(N(\T^n)^{\frac{1}{p}})$, i.e. $N(\pi^n) = o(2^{\frac{n}{p}})$, then we choose the partition $D=(d^n)_{n\in\N}$ to be the dyadic partition $\T$ itself, i.e. set $l_n=n$. Otherwise, we have $\limsup_{n \in \N} \frac{N(\pi^n)^p}{N(\T^n)} > 0$, which implies that $\limsup_{n \in \N} \frac{N(\pi^n)^{p+\varepsilon}}{N(\T^n)} = \infty$. Now we define the increasing subsequence $l_n$ as
  \begin{equation*}
    l_n = \inf\{l\geq n \; : \; N(\pi^n)^{p+\varepsilon} \leq N(\T^l)\}.
  \end{equation*}
  Since $N(\T^l) = 2^l \to \infty$ as $n \to \infty$, finding such an $l_n$ is always possible and since $l_n > n$, we also have $l_n \to \infty$ as $n \to \infty$. So from the construction above, $d^n = \T^{l_n}$ for all $n \in \N$ is indeed a sub-/super-sequence of $\T$ which satisfies $(i)$ of Definition~\ref{def: levy roughness}. Now we will show that condition $(ii)$ of Definition~\ref{def: levy roughness} is also satisfied.

  The generic partition points of $\pi^n$ are denoted as $t^n_k$, and the generic partition points of $d^n$ are denoted as $s^n_j$. Define for $k= 1, 2, \dots, N(\pi^n) - 1$:
  \begin{equation*}
    p(n,k) = \inf\{i \geq 1: \, s^n_m \in (t^n_{k},t^n_{k+1}] \}.
  \end{equation*}
  Then, we get the following inequality regarding the partition points of $\pi^n$ and $d^n$:
  \begin{equation*}
    s^n_{p(n,k)-1} \leq t^n_k < s^n_{p(n,k)}  <  \cdots < s^n_{p(n,k+1)-1} \leq t^n_{k+1} <  s^n_{p(n,k+1)},
  \end{equation*}
  where $p(n,N(\pi^n)) - 1 = N(d^n)$ and $p(n,0) = 1 $.

  Since, by assumption, $\cL^\pi_t(X) = \cL^\T_t(X)$, and $D$ is a sub-/super-sequence of $\T$, we also have that
  \begin{align*}
    &\lim_{n \to \infty} |\cL^{\pi^n}_t(X) - \cL^{d^n}_t(X)| = 0 \\
    \iff &\lim_{n \to \infty}  \sum_{k=0}^{N(\pi^n)-1} (X_{t^n_k} + X_{t^n_{k+1}}) \otimes X_{t^n_i,t^n_{i+1}} - \sum_{k'=0}^{N(d^n)-1} (X_{t^n_{k'}} + X_{t^n_{k'+1}}) \otimes X_{t^n_i,t^n_{i+1}} = 0 \\
    \iff &\lim_{n\to \infty}  \sum_{k=0}^{N(\pi^n)-1} \bigg[ (X_{t^n_k} + X_{t^n_{k+1}}) \otimes X_{t^n_i,t^n_{i+1}} - \sum_{i=p(n,k)}^{p(n,k+1)-1}  (X_{s^{n}_i} + X_{s^{n}_{i+1}}) \otimes X_{s^n_i,s^{n}_{i+1}} \bigg] = 0.
  \end{align*}
  The rest of the proof is very similar to the proof of Part~1. Since all the arguments of the main theorem are reversible, we also obtain L{\'e}vy roughness as a necessary condition.
\end{proof}

\subsection{Invariance of pathwise integration}

As a consequence of the concepts of quadratic roughness and of L{\'e}vy roughness, we can deduce the invariance of the pathwise integral~\eqref{eq: general pathwise integral}. For that propose, recall that $\T = (\T^n)_{n\in \N}$ denotes the sequence of dyadic partitions.

\begin{proposition}
  Let $\gamma \in [0,1]$, $p \in (2,3)$, let $X$ and $\pi = (\pi^n)_{n \in \N}$ be as in Assumption~\ref{ass: partitions}. Suppose that $X$ has quadratic roughness\footnote{For $\gamma = \frac{1}{2}$, the assumption of quadratic roughness is not necessary due to Lemma~\ref{lem: Levy eq to 1/2 RIE}.} and L{\'e}vy roughness with index $p$ along $\pi$, and suppose that $X$ satisfies Property $\gamma$-\textup{(RIE)} relative to $\gamma$, $p$ and $\pi$ as well as to $\T$. Let $q > 0$ be such that $\frac{2}{p} + \frac{1}{q} > 1$ and let $(Y,Y') \in \mathscr{C}^{p,q}_X([0,T];\R^{m \times d})$ be a controlled path. Then,
  \begin{equation*}
    \int_0^t Y_s \dd^{\gamma,\pi} X_s = \int_0^t Y_s \dd^{\gamma,\T} X_s , \qquad t \in [0,T],
  \end{equation*}
  where we refer to \eqref{eq: general pathwise integral} for the definition of these integrals.
\end{proposition}

\begin{proof}
  Since $X$ has quadratic roughness and L{\'e}vy roughness with index~$p$ along $\pi$, the resulting quadratic variation and L{\'e}vy area along $\pi$ and $\T$ coincide by Proposition~\ref{prop: invariance of quad variation} and Theorem~\ref{thm:levy roughness}, and, consequently, the resulting rough paths coincide as well by Remark~\ref{rem:rough path and levy area}. Hence, the corollary follows immediately by Theorem~\ref{thm: general pathwise stochastic integral}.
\end{proof}

\begin{remark}
    Assuming that a path $X$ has quadratic roughness and L{\'e}vy roughness along $\pi$ and assuming that $X$ satisfies Property $\gamma$-(RIE) relative to $\pi$ and $\T$, in particular, the rough paths $\mathbb{X}^{\gamma,\pi}$ and $\mathbb{X}^{\gamma,\T}$ as defined in Proposition~\ref{prop: rough path lift under gamma RIE} do coincide. Differently said, we obtain the invariance of the rough path with respect to the choice of the sequence of partitions.
\end{remark}

\begin{remark}
The quadratic roughness condition and the L{\'e}vy roughness condition imply invariance of quadratic variation and Levy area with respect to the choice of the sequence of partitions. We would like to point out that the partitions in consideration in Section \ref{sec: invariance} are all deterministic partitions, and do not depend on the underlying process (e.g., dyadic Lebesgue partitions). The $p$-th variation along Lebesgue-type partitions often behaves very differently from that of deterministic partitions. For example, see the conjecture in \cite[Section~1.2]{Das2023}, which suggests the $1/H$-th variation of fractional Brownian motion along deterministic partitions (e.g., dyadic partitions) differs from the $1/H$-th variation of fractional Brownian motion along Lebesgue-type partitions.
\end{remark}

\section{Application to stochastic processes and integration}\label{sec: application to stochastic integration}

In this section we demonstrate that the deterministic theory developed in Section~\ref{sec: the rough integral as a limit of general Riemann sums} and ~\ref{sec: invariance} can be applied to the sample paths of various stochastic processes and allows to recover classical stochastic integration like It{\^o} and Stratonovich integration.

\medskip
As shown in~\cite[Sections~3 and~4]{Allan2023c}, almost all sample paths of the following stochastic processes satisfy Property~\textup{(RIE)}:
\begin{itemize}
  \item[$\vcenter{\hbox{\scriptsize$\bullet$}}$] Brownian motion, relative to sequences of equidistant partitions $(\pi^n)_{n \in \N}$ such that $|\pi^n|^{2-\frac{4}{p}} \to 0$ as $n \to \infty$;
  \item[$\vcenter{\hbox{\scriptsize$\bullet$}}$] It{\^o} processes, relative to the dyadic partitions, i.e., $\pi^n = \{k2^{-n}T\}_{k=0}^{2^n}$, $n \in \N$;
  \item[$\vcenter{\hbox{\scriptsize$\bullet$}}$] continuous semimartingales, relative to the sequence of partitions $\pi^n = \{\tau^n_k: k \in \N \cup \{0\}\}$, $n \in \N$, where $\tau^n_0 = 0$, $\tau^n_k = \inf \{t > \tau^n_{k-1}: |t-\tau^n_{k-1}| + |X_t - X_{\tau^n_{k-1}}| \geq 2^{-n} \} \wedge T$, $k \in \N$;
  \item[$\vcenter{\hbox{\scriptsize$\bullet$}}$] the pair $(W,W^H)$, where $W$ denotes the Brownian motion and $W^H$ the fractional Brownian motion with Hurst parameter $H > \frac{1}{2}$, relative to sequences of equidistant partitions $(\pi^n)_{n \in \N}$ such that $(\pi^n)^{2-\frac{4}{p}} \to 0$ as $n \to \infty$;
  \item[$\vcenter{\hbox{\scriptsize$\bullet$}}$] the pair $(\eta, W)$, where $\eta$ denotes a deterministic $\frac{1}{p}$-H{\"o}lder continuous path, relative to the sequence of dyadic partitions.
\end{itemize}
Consequently, due to Lemma~\ref{lem: RIE eq to gamma RIE}, \ref{lem: RIE eq to Levy and quad var} and~\ref{lem: Levy eq to 1/2 RIE}, almost all sample paths of the aforementioned stochastic processes possess quadratic variation and L{\'e}vy area as well as satisfy Property $\gamma$-\textup{(RIE)} relative to any $\gamma \in [0,1]$ and $p \in (2,3)$.

\begin{remark}
  Since $X$ having Property $\gamma$-\textup{(RIE)} relative to $\gamma = \frac{1}{2}$ is equivalent to $X$ possessing L{\'e}vy area, the deterministic theory of Section~\ref{sec: the rough integral as a limit of general Riemann sums} and~\ref{sec: invariance} can also be applied to stochastic processes which do not possess quadratic variation, like the fractional Brownian motion for Hurst parameter $H \in (\frac{1}{3}, \frac{1}{2})$. Indeed, the sample paths of fractional Brownian motion do not possess quadratic variation, see e.g.~\cite{Rogers97,Dudley1998}, and, consequently, they do not satisfy Property \textup{(RIE)}, see also Lemma~\ref{lem: RIE eq to Levy and quad var}. However, it is well-established that almost all sample paths of the fractional Brownian motion possess L{\'e}vy area relative to the dyadic partitions, which follows from the fact that almost all sample paths of fractional Brownian motion are of finite $\frac{1}{\alpha}$-variation for $\alpha < H$ and e.g.~\cite[Theorem~2]{Coutin2002}, where a rough path lift over the fractional Brownian motion was constructed using dyadic approximations.
\end{remark}

To relate the pathwise integral defined in~\eqref{eq: general pathwise integral} to classical stochastic integration like It{\^o} and Stratonovich integration, we consider a $d$-dimensional continuous semimartingale $X=(X_t)_{t\in [0,T]}$, defined on a probability space $(\Omega,\mathcal{F},\P)$ with a filtration $(\mathcal{F}_t)_{t \in [0,T]}$ satisfying the usual conditions, i.e., completeness and right-continuity.

We first consider It{\^o} integration, which corresponds to $\gamma=0$. It is well-known that the semimartingale $X$ can be lifted to a random rough path via It{\^o} integration, see e.g.~\cite{Coutin2005}, by defining $\bX = (X,\X) \in \mathcal{C}^p([0,T];\R^d)$, $\P$-a.s., for any $p \in (2,3)$, where
\begin{equation}\label{eq: Ito rough path}
  \X_{s,t} := \int_s^t (X_{r} - X_s) \otimes \d X_r = \int_s^t X_r \otimes  \d X_r - X_s \otimes X_{s,t}, \qquad (s,t) \in \Delta_T.
\end{equation}

\begin{lemma}\label{lem: Ito integration}
  Let $p \in (2,3)$ and let $\pi = (\pi^n)_{n\in \N}$ be a sequence of adapted partitions with $\pi^n = \{\tau^n_k\}_{k\in \N}$, $n \in \N$, such that each $\tau^n_k$ is a stopping time and such that for almost every $\omega \in \Omega$, $(\pi^n(\omega))_{n \in \N}$ is a sequence of (finite) partitions of $[0,T]$ with vanishing mesh size.

  Let $X$ be a $d$-dimensional continuous semimartingale, and suppose that the sample path $X(\omega)$ satisfies Property $\gamma$-\textup{(RIE)} relative to $\gamma = 0$, $p$ and $(\pi^n(\omega))_{n \in \N}$.
  \begin{enumerate}
    \item[(i)] The random rough paths $\bX = (X,\X)$, with $\X$ defined pathwise via~\eqref{eq: rough path lift via gamma RIE} for $\gamma = 0$, and with $\X$ defined by stochastic integration as in~\eqref{eq: Ito rough path}, coincide $\P$-almost surely.

    \item[(ii)] Let $Y=(Y_t)_{t\in [0,T]}$ be a stochastic process adapted to $(\mathcal{F}_t)_{t\in [0,T]}$. Suppose that, for almost every $\omega \in \Omega$, $(Y(\omega),Y'(\omega))$ is a controlled path in $\mathscr{C}^{p,q}_{X(\omega)}([0,T];\R^{m \times d})$. Then, the rough and It{\^o} integrals of $Y$ against $X$ coincide $\P$-almost surely, that is,
    \begin{equation*}
      \int_0^t Y_s(\omega) \dd^{0,\pi} X_s(\omega) = \int_0^t Y_s(\omega) \dd \bX^0_s(\omega) = \bigg( \int_0^t Y_{s-} \dd X_s \bigg)(\omega)
    \end{equation*}
    holds for every $t \in [0,T]$ and for almost every $\omega \in \Omega$, where $\bX^0(\omega)$ is the canonical rough path lift of $X(\omega)$ as defined via Property $\gamma$-\textup{(RIE)} for $\gamma=0$.
  \end{enumerate}
\end{lemma}

\begin{proof}
  \emph{(i):} By construction, the pathwise rough integral $\int_0^t X_r(\omega) \otimes \d^{0,\pi} X_r(\omega)$ constructed via Property $\gamma$-\textup{(RIE)} for $\gamma = 0$ is given by the limit as $n \to \infty$ of left-point Riemann sums:
  \begin{equation*}
    \sum_{k=0}^{N_n-1} X_{\tau^n_{k+1}(\omega)}\otimes X_{\tau^n_k(\omega) \wedge t, \tau^n_{k+1}(\omega) \wedge t}(\omega).
  \end{equation*}
  It is well-known that these Riemann sums also converge uniformly in probability to the It{\^o} integral $\int_0^t X_r \otimes \d X_r$ and the result thus follows from the (almost sure) uniqueness of limits.

  \emph{(ii):} By, e.g., \cite[Ch.~II, Theorem~21]{Protter2005}, we have that
  \begin{equation*}
    \sum_{k=0}^{N_n-1} Y_{\tau_k^n} X_{\tau_k^n \wedge t, \tau_{k+1}^n \wedge t} \, \longrightarrow \, \int_0^t Y_{s-} \dd X_s \qquad \text{as} \quad n \to \infty,
  \end{equation*}
  where the convergence holds uniformly in $t \in [0,T]$ in probability. By taking a subsequence if necessary, we can then assume that the uniform convergence holds almost surely. On the other hand, by Theorem~\ref{thm: general pathwise stochastic integral}, we know that, for almost every $\omega \in \Omega$,
  \begin{equation*}
    \sum_{k=0}^{N_n-1} Y_{\tau_k^n(\omega)}(\omega) X_{\tau_k^n(\omega) \wedge t, \tau_{k+1}^n(\omega) \wedge t}(\omega) \, \longrightarrow \, \int_0^t Y_s(\omega) \dd \bX_s(\omega) \qquad \text{as} \quad n \to \infty,
  \end{equation*}
  uniformly for $t \in [0,T]$. The result thus follows by the uniqueness of limits and the remaining equality by Theorem~\ref{thm: general pathwise stochastic integral}.
\end{proof}

Next we consider Stratonovich integration, which corresponds to $\gamma=\frac{1}{2}$. It is well-known that the semimartingale $X$ can be lifted to a random rough path via Stratonovich integration, see e.g.~\cite{Coutin2005}, by defining $\bX = (X,\X) \in \mathcal{C}^p([0,T];\R^d)$, $\P$-a.s., for any $p \in (2,3)$, where
\begin{equation}\label{eq: Stratonovich rough path}
  \X_{s,t} := \int_s^t (X_{r} - X_s) \otimes \circ \d X_r = \int_s^t X_r \otimes \circ \d X_r - X_s \otimes X_{s,t}, \qquad (s,t) \in \Delta_T.
\end{equation}

\begin{lemma}\label{lem: Stratonovich integration}
  Let $p \in (2,3)$ and let $\pi = (\pi^n)_{n\in \N}$ be a sequence of adapted partitions with $\pi^n = \{\tau^n_k\}_{k\in \N}$, $n \in \N$, such that each $\tau^n_k$ is a stopping time and such that for almost every $\omega \in \Omega$, $(\pi^n(\omega))_{n \in \N}$ is a sequence of (finite) partitions of $[0,T]$ with vanishing mesh size.

  Let $X$ be a $d$-dimensional continuous semimartingale, and suppose that the sample path $X(\omega)$ satisfies Property $\gamma$-\textup{(RIE)} relative to $\gamma = \frac{1}{2}$, $p$ and $(\pi^n(\omega))_{n \in \N}$.
  \begin{enumerate}
    \item[(i)] The random rough paths $\bX = (X,\X)$, with $\X$ defined pathwise via~\eqref{eq: rough path lift via gamma RIE} for $\gamma = \frac{1}{2}$, and with $\X$ defined by stochastic integration as in~\eqref{eq: Stratonovich rough path}, coincide $\P$-almost surely.

    \item[(ii)] Let $(Y,Y')$ be a continuous semimartingale. Suppose that, for almost every $\omega \in \Omega$, $(Y(\omega), Y'(\omega))$ is a controlled path in $\mathscr{C}^p_{X(\omega)}([0,T];\R^{m \times d})$. Then the rough and Stratonovich integrals of $Y$ against $X$ coincide $\P$-almost surely, that is,
    \begin{equation*}
       \int_0^t Y_s(\omega) \dd^{\frac{1}{2},\pi} X_s(\omega) =\int_0^t Y_s(\omega) \dd \bX^{\frac{1}{2}}_s(\omega) = \Big( \int_0^t Y_s \circ \d X_s \Big) (\omega)
    \end{equation*}
    holds for every $t \in [0,T]$ and for almost every $\omega \in \Omega$, where $\bX^{\frac{1}{2}}(\omega)$ is the canonical rough path lift of $X(\omega)$ as defined in Proposition~\ref{prop: rough path lift under gamma RIE}, using Property $\gamma$-\textup{(RIE)} for $\gamma = \frac{1}{2}$.
  \end{enumerate}
\end{lemma}

\begin{proof}
  \emph{(i):} By construction, the pathwise rough integral $\int_0^t X_r(\omega) \otimes \d^{\frac{1}{2},\pi} X_r(\omega)$ constructed via Property $\gamma$-\textup{(RIE)} for $\gamma = \frac{1}{2}$ is given by the limit as $n \to \infty$ of mid-point Riemann sums:
  \begin{equation*}
    \sum_{k=0}^{N_n-1} \frac{1}{2}(X_{\tau^n_k(\omega)}(\omega) + X_{\tau^n_{k+1}(\omega)}(\omega)) \otimes X_{\tau^n_k(\omega) \wedge t, \tau^n_{k+1}(\omega) \wedge t}(\omega).
  \end{equation*}
  It is well-known that these Riemann sums also converge uniformly in probability to the Stratonovich integral $\int_0^t X_r \otimes \circ \d X_r$ (see e.g.~\cite[Chapter~II, Theorem~21, Theorem~22]{Protter2005}), and the result thus follows from the (almost sure) uniqueness of limits.

  \emph{(ii):} By, e.g.~\cite[Chapter~II, Theorem~21, Theorem~23]{Protter2005}, we have that
  \begin{equation*}
    \sum_{k=0}^{N_n-1} \frac{1}{2}(Y_{\tau^n_k} + Y_{\tau^n_{k+1}}) X_{\tau^n_k \wedge t, \tau^n_{k+1} \wedge t} \longrightarrow \int_0^t Y_s \circ \d X_s \qquad \text{as} \quad n \to \infty,
  \end{equation*}
  where the convergence holds uniformly (in $t \in [0,T]$) in probability. By taking a subsequence if necessary, we can then assume that the (uniform) convergence holds almost surely. On the other hand, by Theorem~\ref{thm: general pathwise stochastic integral}, we know that for almost every $\omega \in \Omega$,
  \begin{equation*}
    \sum_{k=0}^{N_n-1} \frac{1}{2}(Y_{\tau^n_k(\omega)}(\omega) + Y_{\tau^n_{k+1}(\omega)}(\omega)) X_{\tau^n_k(\omega) \wedge t, \tau^n_{k+1}(\omega) \wedge t}(\omega) \longrightarrow \int_0^t Y_s(\omega) \dd \bX^{\frac{1}{2}}_s(\omega) \quad \text{as} \quad n \to \infty
  \end{equation*}
  uniformly in $t \in [0,T]$. The result thus follows by the uniqueness of limits. Hence, the remaining equality follows by Theorem~\ref{thm: general pathwise stochastic integral}.
\end{proof}

\begin{remark}
  Note that the rough path lift of a continuous path uniquely determines the quadratic variation and the L{\'e}vy area, assuming that the rough path is defined as a limit of general Riemann sums for $\gamma \neq \frac{1}{2}$. Indeed, the quadratic variation corresponds to the symmetric part of a rough path and the L{\'e}vy area to the antisymmetric part, cf. Remark~\ref{rem:rough path and levy area}. Hence, Lemma~\ref{lem: Ito integration} and~\ref{lem: Stratonovich integration} imply that the quadratic variation and the L{\'e}vy area of sample paths of continuous semimartingales are, almost surely, independent of the particular sequence of partitions along which they are defined. As seen in Section~\ref{sec: invariance}, the invariance of quadratic variation and L{\'e}vy area is essentially equivalent to the statement that sample paths of continuous semimartingales possess quadratic and L{\'e}vy roughness.
\end{remark}

\bibliographystyle{siam}
\bibliography{references}

\begin{thebibliography}{10}

\bibitem{Allan2023a}
{\sc A.~L. Allan, C.~Cuchiero, C.~Liu, and D.~J. Pr\"omel}, {\em Model-free
  portfolio theory: a rough path approach}, Math. Finance, 33 (2023),
  pp.~709--765.

\bibitem{Allan2023c}
{\sc A.~L. Allan, A.~P. Kwossek, C.~Liu, and D.~J. Prömel}, {\em Pathwise
  convergence of the {E}uler scheme for rough and stochastic differential
  equations}, arXiv preprint arXiv:2309.16489,  (2023).

\bibitem{Allan2024}
{\sc A.~L. Allan, C.~Liu, and D.~J. Pr\"{o}mel}, {\em A c\`adl\`ag rough path
  foundation for robust finance}, Finance Stoch., 28 (2024), pp.~215--257.

\bibitem{Ananova2023}
{\sc A.~Ananova}, {\em Rough differential equations with path-dependent
  coefficients}, Ann. H. Lebesgue, 6 (2023), pp.~1--29.

\bibitem{Ananova2017}
{\sc A.~Ananova and R.~Cont}, {\em Pathwise integration with respect to paths
  of finite quadratic variation}, J. Math. Pures Appl. (9), 107 (2017),
  pp.~737--757.

\bibitem{Bichteler1981}
{\sc K.~Bichteler}, {\em Stochastic integration and {$L\sp{p}$}-theory of
  semimartingales}, Ann. Probab., 9 (1981), pp.~49--89.

\bibitem{Chiu2018}
{\sc H.~Chiu and R.~Cont}, {\em On pathwise quadratic variation for c\`adl\`ag
  functions}, Electron. Commun. Probab., 23 (2018), pp.~Paper No. 85, 12.

\bibitem{Chiu2022}
\leavevmode\vrule height 2pt depth -1.6pt width 23pt, {\em Causal functional
  calculus}, Trans. London Math. Soc., 9 (2022), pp.~237--269.

\bibitem{cont-das2022}
{\sc R.~Cont and P.~Das}, {\em Quadratic variation along refining partitions:
  constructions and examples}, J. Math. Anal. Appl., 512 (2022), pp.~Paper No.
  126173, 35.

\bibitem{contdas2023}
\leavevmode\vrule height 2pt depth -1.6pt width 23pt, {\em Quadratic variation
  and quadratic roughness}, Bernoulli, 29 (2023), pp.~496--522.

\bibitem{Cont2010}
{\sc R.~Cont and D.-A. Fourni\'e}, {\em Change of variable formulas for
  non-anticipative functionals on path space}, J. Funct. Anal., 259 (2010),
  pp.~1043--1072.

\bibitem{Cont2024}
{\sc R.~Cont and R.~Jin}, {\em Fractional {I}to calculus}, Trans. Amer. Math.
  Soc. Ser. B, 11 (2024), pp.~727--761.

\bibitem{Cont2019}
{\sc R.~Cont and N.~Perkowski}, {\em Pathwise integration and change of
  variable formulas for continuous paths with arbitrary regularity}, Trans.
  Amer. Math. Soc. Ser. B, 6 (2019), pp.~161--186.

\bibitem{Coutin2005}
{\sc L.~Coutin and A.~Lejay}, {\em Semi-martingales and rough paths theory},
  Electron. J. Probab., 10 (2005), pp.~no. 23, 761--785.

\bibitem{Coutin2002}
{\sc L.~Coutin and Z.~Qian}, {\em Stochastic analysis, rough path analysis and
  fractional {B}rownian motions}, Probab. Theory Related Fields, 122 (2002),
  pp.~108--140.

\bibitem{Cuchiero2019}
{\sc C.~Cuchiero, W.~Schachermayer, and T.-K.~L. Wong}, {\em Cover's universal
  portfolio, stochastic portfolio theory, and the num\'eraire portfolio}, Math.
  Finance, 29 (2019), pp.~773--803.

\bibitem{Das2023}
{\sc P.~Das, R.~{\L}ochowski, T.~Matsuda, and N.~Perkowski}, {\em Level
  crossings of fractional {B}rownian motion}, Accepted for publication in
  Annals of Probability,  (2023).

\bibitem{Davis2014}
{\sc M.~Davis, J.~Ob\l\'oj, and V.~Raval}, {\em Arbitrage bounds for prices of
  weighted variance swaps}, Math. Finance, 24 (2014), pp.~821--854.

\bibitem{Davis2018}
{\sc M.~Davis, J.~Ob\l\'oj, and P.~Siorpaes}, {\em Pathwise stochastic calculus
  with local times}, Ann. Inst. Henri Poincar\'e{} Probab. Stat., 54 (2018),
  pp.~1--21.

\bibitem{Dudley1998}
{\sc R.~Dudley and R.~Norvai{\v{s}}a}, {\em An Introduction to p-Variation and
  Young Integrals: With Emphasis on Sample Functions of Stochastic Processes},
  vol.~1 of MaPhySto Lecture notes, Aarhus, Denmark, 1998.

\bibitem{Dupire2019}
{\sc B.~Dupire}, {\em Functional {I}t\^o{} calculus}, Quant. Finance, 19
  (2019), pp.~721--729.

\bibitem{Follmer1981}
{\sc H.~F\"ollmer}, {\em Calcul d'{I}t\^o{} sans probabilit\'es}, in Seminar on
  {P}robability, {XV} ({U}niv. {S}trasbourg, {S}trasbourg, 1979/1980)
  ({F}rench), vol.~850 of Lecture Notes in Math., Springer, Berlin, 1981,
  pp.~143--150.

\bibitem{Follmer2013}
{\sc H.~F\"ollmer and A.~Schied}, {\em Probabilistic aspects of finance},
  Bernoulli, 19 (2013), pp.~1306--1326.

\bibitem{Freedman1983}
{\sc D.~Freedman}, {\em Brownian motion and diffusion}, Springer-Verlag, New
  York-Berlin, second~ed., 1983.

\bibitem{Friz2020}
{\sc P.~K. Friz and M.~Hairer}, {\em A course on rough paths}, Universitext,
  Springer, Cham, second~ed., 2020.
\newblock With an introduction to regularity structures.

\bibitem{Friz2010}
{\sc P.~K. Friz and N.~B. Victoir}, {\em Multidimensional stochastic processes
  as rough paths}, vol.~120 of Cambridge Studies in Advanced Mathematics,
  Cambridge University Press, Cambridge, 2010.
\newblock Theory and applications.

\bibitem{Friz2018}
{\sc P.~K. Friz and H.~Zhang}, {\em Differential equations driven by rough
  paths with jumps}, J. Differential Equations, 264 (2018), pp.~6226--6301.

\bibitem{Imkeller2015}
{\sc P.~Imkeller and D.~J. Pr\"omel}, {\em Existence of {L}\'evy's area and
  pathwise integration}, Commun. Stoch. Anal., 9 (2015), pp.~93--111.

\bibitem{Karandikar1995}
{\sc R.~L. Karandikar}, {\em On pathwise stochastic integration}, Stochastic
  Process. Appl., 57 (1995), pp.~11--18.

\bibitem{Levy1940}
{\sc P.~L\'evy}, {\em Le mouvement brownien plan}, Amer. J. Math., 62 (1940),
  pp.~487--550.

\bibitem{Lyons1995}
{\sc T.~J. Lyons}, {\em {Uncertain volatility and the risk-free synthesis of
  derivatives}}, Appl. Math. Finance, 2 (1995), pp.~117--133.

\bibitem{Lyons1998}
\leavevmode\vrule height 2pt depth -1.6pt width 23pt, {\em Differential
  equations driven by rough signals}, Rev. Mat. Iberoamericana, 14 (1998),
  pp.~215--310.

\bibitem{Lyons2007}
{\sc T.~J. Lyons, M.~Caruana, and T.~L\'evy}, {\em Differential equations
  driven by rough paths}, vol.~1908 of Lecture Notes in Mathematics, Springer,
  Berlin, 2007.
\newblock Lectures from the 34th Summer School on Probability Theory held in
  Saint-Flour, July 6--24, 2004, With an introduction concerning the Summer
  School by Jean Picard.

\bibitem{Nutz2012}
{\sc M.~Nutz}, {\em Pathwise construction of stochastic integrals}, Electron.
  Commun. Probab., 17 (2012), pp.~no. 24, 7.

\bibitem{Perkowski2016}
{\sc N.~Perkowski and D.~J. Pr\"{o}mel}, {\em Pathwise stochastic integrals for
  model free finance}, Bernoulli, 22 (2016), pp.~2486--2520.

\bibitem{Protter2005}
{\sc P.~E. Protter}, {\em Stochastic integration and differential equations},
  vol.~21 of Stochastic Modelling and Applied Probability, Springer-Verlag,
  Berlin, 2005.
\newblock Second edition. Version 2.1, Corrected third printing.

\bibitem{Rogers97}
{\sc L.~C.~G. Rogers}, {\em Arbitrage with fractional {B}rownian motion}, Math.
  Finance, 7 (1997), pp.~95--105.

\bibitem{Schied2018}
{\sc A.~Schied, L.~Speiser, and I.~Voloshchenko}, {\em Model-free portfolio
  theory and its functional master formula}, SIAM J. Financial Math., 9 (2018),
  pp.~1074--1101.

\end{thebibliography}

\end{document}